%% file: arxiv_review.tex
\documentclass[smallextended,natbib,runningheads]{svjour3}
\journalname{Annals of the Institute of Statistical Mathematics}
\smartqed

\usepackage[utf8]{inputenc}
\usepackage{amssymb}
\usepackage{amsmath}
\usepackage{url}
\usepackage{graphicx}
\usepackage{color}
\usepackage[colorlinks=true,linkcolor=blue,citecolor=blue,pdfborder={0 0 0}]{hyperref}
\usepackage{bm}
\usepackage[]{mdframed}
\usepackage{hyperref}
\usepackage{enumerate}
\usepackage{subcaption}
\usepackage[normalem]{ulem}
\usepackage{tikz-cd}
\usepackage{mathdots}

\newcommand{\change}[1]{{\color{black}{#1}}}%{{\color{magenta}{#1}}}

\newcommand{\N}{\mathbb{N}}
\newcommand{\R}{\mathbb{R}}
\newcommand{\Bc}{\mathcal{B}}

\newcommand{\Mc}{\mathcal{M}}
\newcommand{\Nc}{\mathcal{N}}
\newcommand{\Sc}{\mathcal{S}}
\newcommand{\Xc}{\mathcal{X}}
\newcommand{\Yc}{\mathcal{Y}}
\newcommand{\Hc}{\mathcal{H}}
\newcommand{\Ic}{\mathcal{I}}

\newcommand{\1}{\mathbf{1}}
\renewcommand{\Pr}{\mathbb{P}}

\newcommand{\supp}{\operatorname{supp}}
\newcommand{\un}[1]{\underline{#1}}
\newcommand{\ol}[1]{\overline{#1}}
\newcommand{\vect}{\mathrm{vec}}
\newcommand{\phidivergence}{$\phi$\nobreakdash-di\-ver\-gen\-ce}
\newcommand{\phidivergences}{$\phi$\nobreakdash-di\-ver\-gen\-ces}

\DeclareMathOperator*{\arginf}{arg\,inf}
\DeclareMathOperator*{\diag}{diag}
\DeclareMathOperator*{\im}{Im}

\let\emptyset\varnothing

\newtheorem{prop}{Proposition}%[section]
\newtheorem{defn}[prop]{Definition}
\newtheorem{thm}[prop]{Theorem}
\newtheorem{cor}[prop]{Corollary}
\newtheorem{lem}[prop]{Lemma}
\newtheorem{cond}[prop]{Condition}

\spnewtheorem*{xproof}{}{\itshape}{\rmfamily}% the label is assigned later

\renewenvironment{proof}[1][\proofname]
{\renewcommand\xproofname{#1}\xproof}
{\endxproof}

\begin{document}

\title{On the differentiability of $\phi$-projections in the discrete finite case}

\author{Gery Geenens \and Ivan Kojadinovic \and Tommaso Martini}

\institute{
Gery Geenens \at
School of Mathematics and Statistics, UNSW Sydney, Australia \\
\email{ggeenens@unsw.edu.au}  
\and
Ivan Kojadinovic \at
Laboratoire de math\'ematiques et de leurs applications, CNRS / Universit\'e de Pau et des Pays de l'Adour, France \\
\email{ivan.kojadinovic@univ-pau.fr}    
\and
Tommaso Martini \at
Laboratoire de math\'ematiques et de leurs applications, CNRS / Universit\'e de Pau et des Pays de l'Adour, France, and Department of Mathematics G.\ Peano, Università degli Studi di Torino, Italy \\
\email{tommaso.martini@unito.it}
}

\date{} %\date{Received: date / Revised: date}

\maketitle

\begin{abstract}
  In the case of finite measures on finite spaces, we state conditions under which $\phi$-projections are continuously differentiable. When the set on which one wishes to $\phi$-project is convex, we show that the required assumptions are implied by easily verifiable conditions. In particular, for input probability vectors and a rather large class of \phidivergences, we obtain that $\phi$-projections are continuously differentiable when projecting on a set defined by linear equalities. The obtained results are applied to $\phi$-projection estimators (that is, minimum \phidivergence\ estimators). \change{A first application, rooted in robust statistics, concerns the computation of the influence functions of such estimators. In a second set of applications, we derive their asymptotics} when projecting on parametric sets of probability vectors, on sets of probability vectors generated from distributions with certain moments fixed and on Fréchet classes of bivariate probability arrays. The resulting asymptotics hold whether the element to be $\phi$-projected belongs to the set on which one wishes to $\phi$-project or not.
  
%{\it MSC 2010:} 62E20, 62G20.
\end{abstract}

%\tableofcontents

\section{Introduction}

The concept of $I$-projection due to Csisz\'ar \citep[see, e.g.,][]{Csi75, CsiShi04} plays a crucial role in many areas of probability and statistics. Informally, given a probability distribution $q_0$ of interest and a set $\Mc$ of probability distributions, it consists of finding an element of $\Mc$, if it exists, that is the ``closest'' to $q_0$ in the sense of the Kullback--Leibler (or information) divergence. A well-known class of alternatives to the Kullback--Leibler divergence (containing the latter) are the so-called \phidivergences\ \citep[see, e.g.,][and the references therein]{AliSil66,Csi67,LieVaj87,CsiShi04} and $\phi$-projections are merely the analogs of $I$-projections based on \phidivergences.

Restricting attention to (probability) measures on finite spaces, the main purpose of this work is to provide conditions under which $\phi$-projections are continuously differentiable with respect to the input distribution. From the point of view of statistical inference, such results allow one to immediately obtain \change{the influence function,} the consistency and the limiting distribution of a $\phi$-projection estimator (also known as minimum \phidivergence\ estimator in the context of parametric inference). \change{For instance}, assume that $q_n$ is a consistent estimator of a probability vector $q_0$ of interest and that the latter can be uniquely $\phi$-projected on a set $\Mc$ of probability vectors of interest. Then, under conditions that imply the continuity of $\phi$-projections, the consistency of the $\phi$-projection of $q_n$ on $\Mc$ is merely a consequence of the continuous mapping theorem and, under conditions that imply the continuous differentiability of $\phi$-projections, the limiting distribution of a properly scaled version of the $\phi$-projection of $q_n$ on $\Mc$ immediately follows from the delta method.

In the considered discrete finite setting, first differentiability results were obtained for $I$-projections on Fréchet classes of bivariate probability arrays by \citet[Section~2.4]{KojMar24} by exploiting, among other things, continuity results for $\phi$-projections on convex sets obtained by \cite{GieRef13,GieRef17}. When the set $\Mc$ of probability vectors on which one wishes to $\phi$-project is not necessarily convex, related results (that are not differentiability results \emph{per se}) were obtained by \citet[Section~2]{JimPinAlbMor11}. When $\Mc$ is a set of probability vectors obtained from a given parametric distribution and $q_0$ is assumed to belong to $\Mc$, the previous results entail the well-known asymptotics of minimum \phidivergence\ estimators \citep[see, e.g.,][and the reference therein]{ReaCre88, MorParVaj95, BasShiPar11}. Since the results of \cite{JimPinAlbMor11} were actually derived regardless of whether $q_0$ belongs to $\Mc$ or not (that is, also potentially under model misspecification), they open the path to the development of many interesting inference procedures related to goodness-of-fit testing and model selection as discussed, for example, in Sections~3 and 4 of \cite{JimPinAlbMor11} (see also the references therein).

The specific aim of this work is to attempt to unify and extend those previous results. Following \cite{GieRef13,GieRef17}, we shall deal with finite measures on finite spaces (and thus not only with probability measures on finite spaces). Then, under a slight strengthening of the main condition considered in \cite{JimPinAlbMor11} and under a key additional assumption missing from that reference (see Remark~\ref{rem:invertibility} in Section~\ref{sec:param}), we shall first state a general result on the differentiability of $\phi$-projections. Under an additional assumption of convexity of the set~$\Mc$, we shall show that the above conditions are implied by simpler ones particularly amenable to verification. For example, we obtain that that for many common choices of \phidivergences, $\phi$-projections are automatically continuously differentiable when~$\Mc$ is a subset defined by linear equalities. Our proofs rely, among others, on results of \cite{GieRef17} and \cite{Rus87}, as well as on the fact that \phidivergences\, constructed from strongly convex functions $\phi$ (in a certain sense), are strongly convex in their first argument.

Note that when one wishes to study the asymptotics of $\phi$-projection and/or minimum \phidivergence\ estimators outside of the finite discrete setup, another line of research consists of using the dual representation of \phidivergences\, \cite[see, e.g.,][]{Kez03, BroKez06, BroKez09}. \change{In the continuous case, this way of proceeding has been successfully used for estimating semi-parametric models  \citep{BroKez10} and for studying the robustness of minimum \phidivergence\ estimators \citep{TomBro11}. As noted by a reviewer, such investigations could also be based on the theory of constrained M-estimators \cite[see, e.g.,][]{Gey94}. We do not elaborate further on these more complex approaches as they do not seem necessary in the considered discrete finite setting, particularly because, as it will become clearer from Section~\ref{sec:framework},  our results are developed within a parametric framework analog to the one adopted, for instance, in \cite{JimPinAlbMor11}}.

The paper is organized as follows. In the first section, we define \phidivergences\, for finite measures on finite spaces, recall their main properties stated in \cite{GieRef17}, provide conditions under which they are strongly convex in their first argument and define $\phi$-projections. Next, we state conditions under which $\phi$-projections are continuous and continuously differentiable and we show that these can be replaced by substantially simpler conditions when the set $\Mc$ on which one wishes to $\phi$-project is convex. Finally, we illustrate how the \change{obtained results can be used to compute the influence functions of minimum \phidivergence\ estimators and derive their asymptotics} when projecting on parametric sets of probability vectors, on sets of probability vectors generated from distributions with certain moments fixed and on Fréchet classes of bivariate probability arrays.

\section{Preliminaries on \phidivergences\, and $\phi$-projections}
\label{sec:prelim}

\subsection{Notation}
%\label{sec:notation}

Let $m \in \N^+$ be fixed and let $\Xc = \{x_1,\dots,x_m\}$ be a set of interest. We shall study \phidivergences\, and $\phi$-projections for finite measures on $(\Xc,2^\Xc)$. Clearly, any finite measure $\mu$ on $(\Xc,2^\Xc)$ is discrete and can be represented by a collection of $m$ point masses since
$$
\mu(A) = \sum_{x_i \in A} \mu(\{x_i\}), \qquad A \subset \{x_1,\dots,x_m\}.
$$
As shall become clear from the next subsections, the forthcoming developments depend only on $m = |\Xc|$ and the values that finite measures on $\Xc$ take on subsets of $\Xc$, that is, on the singletons $\{x_1\}$, \dots, $\{x_m\}$. For that reason, in the context of this work, a finite measure $\mu$ on $\Xc$ will be simply represented by a vector $s \in [0,\infty)^m$ defined by $s_i = \mu(\{x_i\})$, $i \in \{1,\dots,m\}$. We will use the letters $s,t$ for denoting such point mass vectors in the case of general finite measures, and reserve the letters $p,q$ for point mass vectors of probability measures, that is, whose components sum up to one -- these will be simply referred to as probability vectors in the rest of this work. We define the support of any vector $s \in [0,\infty)^m$ as $\supp(s) = \{ i \in \{1,\dots,m\} : s_i > 0 \}$. Finally, for any set $S \subset [0,\infty)^m$ and any $\Ic \subset \{1,\dots,m\}$, $S_\Ic = \{s \in S : \supp(s) \subset \mathcal{I} \}$. For instance, given $t \in [0,\infty)^m$, $[0,\infty)^m_{\supp(t)}$ denotes the set $\{s \in [0,\infty)^m : \supp(s) \subset \supp(t) \}$, that is, the set of nonnegative vectors whose support is included in that of $t$.

\subsection{\phidivergences\, for finite measures on finite spaces}

Let $\phi:[0,\infty) \to \R$ be a convex function that is continuous at 0. Then, let $f:[0,\infty)^2 \to \R \cup \{\infty\}$ be the function defined by
\begin{equation}
  \label{eq:f}
  f(v,w) =
  \begin{cases}
    \displaystyle w \phi \left( \frac{v}{w} \right) & \text{if } v \in [0,\infty), \, w \in (0,\infty), \\
    \\
    \displaystyle v \lim_{x \to \infty} \frac{\phi(x)}{x} & \text{if } v \in [0,\infty), \, w = 0, \\
    \\
    0 &  \text{if } v = 0, \, w = 0.
  \end{cases}
\end{equation}
Note that, by Proposition A.24 of \cite{LieVaj87}, $\lim_{x \to \infty} \phi(x)/x$ exists in $\R \cup \{\infty\}$ and that, by Proposition A.35 in the same reference, the function $f$ is lower semicontinuous on $[0,\infty)^2$. Following \citet[Section 2]{GieRef17}, it can further be verified that, for any fixed $v \geq 0$, $f(v,\cdot)$ is continuous.

The following additional condition on $\phi$ is common in the literature.

\begin{cond}[Strict convexity of $\phi$]
  \label{cond:phi:strict:convex}
  The function $\phi$ is strictly convex.
\end{cond}

It is easy to verify that, for any fixed $w > 0$, $f(\cdot,w)$ is strictly convex if and only if Condition~\ref{cond:phi:strict:convex} holds.

\begin{defn}[\phidivergence]
For any $s,t \in [0,\infty)^m$, the \phidivergence\ of $s$ relative to $t$ is defined by
\begin{equation}
  \label{eq:phi:div}
D_\phi(s \mid t) = \sum_{i=1}^m f(s_i,t_i).
\end{equation}
\end{defn}

It can verified from~\eqref{eq:f} that, if $\supp(s) \subset \supp(t)$, $D_\phi(s \mid t) < \infty$. Note that because $s$ and $t$ are not necessarily probability vectors, $D_\phi(s \mid t)$ can be negative (take for instance $\phi(x) = x \log(x)$, $x \in (0,\infty)$, $\phi(0) = 0$ and $s_i \leq t_i$, $i \in \{1,\dots,m\}$, with at least one $s_i < t_i$). Furthermore, the following properties \cite[see][Theorem 2.1]{GieRef17} hold as a consequence of the properties of the function $f$ in~\eqref{eq:f} discussed above.

\begin{prop}
  \label{prop:properties:D:phi}
  The following statements hold:
  \begin{enumerate}[(i)]
  \item The function $D_\phi$ is lower semicontinuous on $[0,\infty)^m \times [0,\infty)^m$.
  \item For fixed $s \in [0,\infty)^m$, $D_\phi(s \mid \cdot)$ is continuous on $[0,\infty)^m$.
  \item For fixed $t \in [0,\infty)^m$, $t \neq 0_{\R^m}$, $D_\phi(\cdot \mid t)$ is strictly convex on $[0,\infty)^m_{\supp(t)}$ if and only if Condition~\ref{cond:phi:strict:convex} holds.
\end{enumerate}
\end{prop}

\subsection{Existence of $\phi$-projections on compact subsets}

\begin{defn}[$\phi$-projection]
 % \label{defn:phi:proj}
Let $t \in [0,\infty)^m$ and let $\Mc \subset [0,\infty)^m$, $\Mc \neq \emptyset$. A $\phi$-projection of $t$ on $\Mc$, if it exists, is an element $s^* \in \Mc$ satisfying
\begin{equation*}
  \label{eq:phi:proj}
  D_\phi(s^* \mid t) = \inf_{s \in \Mc} D_\phi(s \mid t).
\end{equation*}
\end{defn}

The following result \citep[see][Lemma~3.1~(i)]{GieRef17} shows that $\phi$-projections exist as soon as the set $\Mc$ on which one wishes to $\phi$-project is compact.

\begin{prop}[Existence of $\phi$-projection]
  \label{prop:existence:phi:proj}
  Let $t \in [0,\infty)^m$ and let $\Mc$ be a nonempty compact subset of $[0,\infty)^m$. Then, there exists a $\phi$-projection of $t$ on $\Mc$.
\end{prop}

We end this subsection by noting that, if one intends to $\phi$-project only on compact sets of nonnegative vectors, without loss of generality, it is enough to study $\phi$-projections of vectors of $[0,1]^m$ on compact subsets of $[0,1]^m$. To see this, let $t \in [0,\infty)^m$ and let $\Mc$ be a nonempty compact subset of $[0,\infty)^m$. Then, there exists $M \in (0,\infty)$ such $t \in [0,M]^m$ and $\Mc \subset [0,M]^m$. Let $t' = t/M \in [0,1]^m$ and let $\Mc' = \{s/M : s \in \Mc \} \subset [0,1]^m$. Clearly, $\Mc'$ is also compact. Furthermore,
\begin{equation*}
  %\label{eq:restriction}
  \arginf_{s' \in \Mc'} D_\phi(s' \mid t') = \frac{1}{M} \arginf_{s \in \Mc} D_\phi(s/M \mid t') =  \frac{1}{M} \arginf_{s \in \Mc} D_\phi(s \mid t),
\end{equation*}
where the second equality is a consequence of the fact that, from~\eqref{eq:f} and~\eqref{eq:phi:div}, for any $s,t \in [0,M]^2$,
$$
D_\phi(s \mid t)  = \sum_{i=1}^m f(s_i,t_i) = M \sum_{i=1}^m f(s_i/M,t_i/M) = M D_\phi(s/M \mid t/M).
$$
Hence, as we shall only be interested in $\phi$-projecting on compact sets of nonnegative vectors, without loss of generality, we shall only consider vectors in $[0,1]^m$ in the remainder of this work.

\subsection{Unicity of $\phi$-projections on compact and convex sets}

Another common condition on $\phi$ appearing in the literature is the following.

\begin{cond}
  \label{cond:phi:lim}
  The function $\phi$ satisfies $\lim_{x \to \infty} \frac{\phi(x)}{x} = \infty$.
\end{cond}

As a direct corollary of Lemma~3.1 of \cite{GieRef17} (with a slight modification of statement (ii) coming from its proof), we have:

\begin{prop}[Unicity of $\phi$-projection]
  \label{prop:phi:proj:convex}
  Let $t \in [0,1]^m$ and $\Mc$ be a nonempty closed and convex subset of $[0,1]^m$. Then, under Condition~\ref{cond:phi:strict:convex}:
  \begin{enumerate}[(i)]
  \item If $\Mc \subset [0,1]^m_{\supp(t)}$, the $\phi$-projection of $t$ on $\Mc$ is unique.
  \item If $\Mc \cap [0,1]^m_{\supp(t)} \neq \emptyset$ and Condition~\ref{cond:phi:lim} holds, the $\phi$-projection of $t$ on $\Mc$ is unique and coincides with the $\phi$-projection of $t$ on $\Mc \cap [0,1]^m_{\supp(t)}$.
  \end{enumerate}
\end{prop}

In the previous proposition, Condition~\ref{cond:phi:strict:convex} is essential: as illustrated in Example 3.2 of \cite{GieRef17}, if $\phi$ is convex but not strictly convex, unicity of $\phi$-projections on compact and convex sets may not hold.

\subsection{\phidivergences\, strongly convex in their first argument}
%\label{sec:strong:convex}

We end this section by stating conditions under which \phidivergences\, are strong\-ly convex in their first argument. Let $\Yc$ be a convex subset of $\R^d$ ($d \in \N^+$) and recall that a function $H$ from $\Yc$ to $\R$ is strongly convex if and only if there exists a constant $\kappa_H \in (0,\infty)$ such that, for any $x,y \in \Yc$ and $\alpha \in [0,1]$,
$$
H(\alpha x + (1-\alpha)y) \leq \alpha H(x) + (1-\alpha) H(y)  - \frac{\kappa_H}{2} \alpha (1-\alpha) \| x - y\|_2^2.
$$
When $H$ is twice continuously differentiable on $\mathring \Yc$, the interior of $\Yc$, we have \citep[see, e.g.,][Theorem~2.1.11]{Nes04} that $H$ is strongly convex on $\Yc$ if and only if, for any $x \in \mathring \Yc$, $\Hc(x) - \kappa_H I_d$ is positive semi-definite, where $\Hc(x)$ is the Hessian matrix of $H$ at $x$ and $I_d$ is the $d \times d$ identity matrix. We shall exploit this equivalent characterization in Section~\ref{sec:convex}.

The following condition on $\phi$ will prove essential in the developments below.

\begin{cond}[Strong convexity of $\phi$]
  \label{cond:phi:strong:convex}
For any $w \in (0,\infty)$, the restriction of $\phi$ to $[0,1/w]$ is strongly convex.
\end{cond}

Note that many functions $\phi$ involved in the definition of classical \phidivergences\, satisfy this condition. The equivalent characterization of strong convexity mentioned above allows the validity of Condition~\ref{cond:phi:strong:convex} to be easily verified for those functions $\phi$ satisfying:
\begin{cond}[Differentiability of $\phi$]
  \label{cond:phi:diff}
  The function $\phi$ is twice continuously differentiable on $(0,\infty)$.
\end{cond}

\renewcommand{\arraystretch}{2}
\begin{table}[t!]
  \centering
  \caption{The function $\phi$ of several classical \phidivergences\, satisfying Conditions~\ref{cond:phi:strong:convex} and~\ref{cond:phi:diff} and the associated strong convexity constant.}
  \label{tab:phi}
    \begin{tabular}{lll}
      \hline
      \hline
      \textbf{Divergence} & $\phi(x)$ &$\kappa_{\phi}(w)$ \\
      \hline
      \hline
      Kullback--Leibler & $x \log x$ & $w$ \\
      Pearson's $\chi^2$ & $(x-1)^2$ & 2 \\
      Squared Hellinger & $2(1-\sqrt{x})$ & $\frac{w^{\frac{3}{2}}}{2}$ \\
      Reverse relative entropy & $-\log x$ & $w^2$ \\
      Vincze--Le Cam & $\frac{(x-1)^2}{x+1}$ & $8 \big(\frac{1}{w}+1 \big)^{-3}$ \\
      Jensen--Shannon & $(x+1) \log \frac{2}{x+1}+x \log x$ & $\frac{w^2}{w+1}$ \\
      Neyman's $\chi^2$ & $\frac{1}{x}-1$ & $2 w^3$ \\
      $\alpha$-divergence & $\frac{4 \big(1-x^{\frac{1+\alpha}{2}} \big)}{1-\alpha^2}, \alpha<3 , \alpha \neq \pm 1$ & $w^{\frac{3-\alpha}{2}}$ \\
      \hline
      \hline
    \end{tabular}
\end{table}
\renewcommand{\arraystretch}{1.3}

Indeed, Condition~\ref{cond:phi:strong:convex} is then equivalent to $\phi''(x) \geq \kappa_\phi(w)$ for all $x \in (0,1/w)$, where $\kappa_\phi(w)$ is the strong convexity constant of the restriction of $\phi$ to $[0,1/w]$. Note that \cite{Mel20} considered \phidivergences\, constructed from such functions $\phi$, but did not investigate their strong convexity in their first argument. Table~\ref{tab:phi} shows the function~$\phi$ and the associated strong convexity constant for several classical \phidivergences\, satisfying Condition~\ref{cond:phi:diff}.

The following result is proven in Section~\ref{proofs:prelim}.

\begin{prop}
  \label{prop:D:phi:strong:convex}
  Under Condition~\ref{cond:phi:strong:convex}, the function $D_\phi(\cdot \mid t)$ is strongly convex on $[0,1]^m_{\supp(t)}$ for all $t \in [0,1]^m$, $t \neq 0_{\R^m}$.
\end{prop}

We end this subsection with the statement of a short lemma, proven in Section~\ref{proofs:prelim}, that we shall use in Section~\ref{sec:convex}.

\begin{lem}
  \label{lem:phi:strict:convex}
  Condition~\ref{cond:phi:strong:convex} implies Condition~\ref{cond:phi:strict:convex}.
\end{lem}

\section{On the continuous differentiability of $\phi$-projections}
\label{sec:diff}

\subsection{Framework}
\label{sec:framework}

Let $t_0 \in [0,1]^m$ be the nonnegative vector to be $\phi$-projected on a nonempty subset $\Mc$ of $[0,1]^m$. Since we are interested in stating differentiability results at $t_0$, we naturally impose that $t_0 \in (0,1)^m$. The aim of this section is to provide conditions under which the function 
	\begin{equation}
		\label{eq:S:*:map}
		\Sc^*(t) = \arginf_{s \in \Mc} D_\phi(s \mid t)
	\end{equation}
is well-defined over an open neighborhood $\Nc(t_0) \subset (0,1)^m$ of $t_0$ and continuously differentiable at $t_0$.

In the rest of this work, we shall assume that the set $\Mc$ is constructed as follows.

\begin{cond}[Construction of $\Mc$]
  \label{cond:M}
  There exists a bounded subset $\Theta$ of $\R^k$ for some strictly positive integer $k \leq m$ with $\mathring \Theta \neq \emptyset$ and a continuous injective function $\Sc$ from $\bar \Theta$ to $[0,1]^m$ that is twice continuously differentiable on $\mathring \Theta$ such that $\Mc = \Sc(\bar \Theta)$ and $\Sc(\mathring \Theta) \subset (0,1)^m$.
\end{cond}

\begin{remark}
  %\label{rem:cond:M}
  Assuming  $\Theta$ bounded in Condition~\ref{cond:M} is not restrictive, as an unbounded parameter space can always be reparametrized using a suitable bijection into a bounded subset of $\R^k$. Also, we do not define $\Theta$ as a closed set as we find it more explicit to write $\bar \Theta$ when necessary. Assuming $\Sc$ injective is particularly natural in a statistical framework as it guarantees the identifiability of the ``parametric'' model $\Sc(\bar \Theta) = \{ \Sc(\theta) : \theta \in \bar \Theta\}$. Furthermore, note that Condition~\ref{cond:M} implies that $\Mc = \Sc(\bar \Theta)$ is compact (since $\bar \Theta$ is compact and $\Sc$ is continuous) and nonempty (because $\mathring \Theta \neq \emptyset$). Hence, by Proposition~\ref{prop:existence:phi:proj}, for any $t \in [0,1]^m$, a $\phi$-projection of $t$ on $\Mc$ exists (but may not be unique). Finally, the assumption that $\Sc(\mathring \Theta) \subset (0,1)^m$ is necessary, on the one hand, so that, for any $t \in (0,1)^m$, the function $D_\phi(\Sc(\cdot) \mid t)$ is twice continuously differentiable on $\mathring \Theta$, and, on the other hand, so that $\Mc \cap (0,1)^m \neq \emptyset$ (which are both needed to obtain the desired differentiability results as we shall see in the next subsections). \qed
\end{remark}

The previous setting is very general. To illustrate this, consider for a moment that the vectors of interest in $[0,1]^m$ are probability vectors. Then, for some $k < m$, $\Mc = \Sc(\bar \Theta)$ could represent a parametric family of probability vectors. For example, the first application in \change{Section~\ref{sec:app:asym}} considers the class of probability vectors obtained from the binomial distributions with parameters $m-1$ and $\theta \in \bar \Theta = [0,1]$. Some thought reveals that most families of discrete distributions with finite support could actually be used. Note, however, that $\Mc$ is typically not convex in such parametric situations. As shall be illustrated in the second and third application in \change{Section~\ref{sec:app:asym}}, Condition~\ref{cond:M} can also be tailored to scenarios in which one wishes to $\phi$-project on a set $\Mc$ of probability vectors defined by linear equalities, a common situation in applications implying the convexity of $\Mc$. In such convex cases, the key conditions necessary for the differentiability of $\phi$-projections are substantially simpler and automatically verified for a rather large class of \phidivergences. This will be established in Section~\ref{sec:convex}. Before that, Section~\ref{sec:param} will first state general conditions under which $\Sc^*$ in~\eqref{eq:S:*:map} is well-defined and continuously differentiable at $t_0$. In what follows, Conditions~\ref{cond:phi:strict:convex},~\ref{cond:phi:diff} and~\ref{cond:M} will always be assumed to hold.

\subsection{The general case of $\Mc$ compact}
\label{sec:param}

We shall first investigate the differentiability of the function $\Sc^*$ in~\eqref{eq:S:*:map} at $t_0 \in (0,1)^m$ without assuming that the set $\Mc = \Sc(\bar \Theta)$ is convex. Specifically, we consider the following (nested) conditions.

\begin{cond}%[Existence and unicity of $\phi$-projections in a neighborhood of $t_0$ -- I]
  [$\phi$-projections in a neighborhood of $t_0$ I]
  \label{cond:unicity:t0}
  There exists an open neighborhood $\Nc(t_0) \subset (0,1)^m$ of $t_0$ such that, for any $t \in \Nc(t_0)$, the $\phi$-projection $s^*$ of $t$ on $\Mc = \Sc(\bar \Theta)$ (exists and) is unique.
\end{cond}

\begin{cond}%[Existence and unicity of $\phi$-projections in a neighborhood of $t_0$ -- II]
  [$\phi$-projections in a neighborhood of $t_0$ II]
  \label{cond:unicity:interior:t0}
Condition~\ref{cond:unicity:t0} holds and, for any $t \in \Nc(t_0)$, the unique $\phi$-projection $s^*= \Sc^*(t)$ satisfies $s^* = \Sc(\theta^*)$ for some (unique) $\theta^* \in \mathring \Theta$.
\end{cond}

Clearly, Condition~\ref{cond:unicity:interior:t0} is a strengthening of Condition~\ref{cond:unicity:t0}, as it requires $\theta^*$ to belong to the interior of $\Theta$. Also, when restricted to probability vectors, Condition~\ref{cond:unicity:interior:t0} is a slight strengthening of Assumption~3 of \cite{JimPinAlbMor11} which is recovered by setting $\Nc(t_0) = \{t_0\}$. That reference notes that this type of condition is common in the literature \cite[see, e.g.,][]{Whi82,Lin94,BroKez09}.

We shall first determine additional conditions under which the function $\vartheta^*:\Nc(t_0) \to \bar \Theta$ defined by
\begin{equation}
  \label{eq:theta:map}
  \vartheta^*(t) = \arginf_{\theta \in \bar \Theta} D_\phi(\Sc(\theta) \mid t)
\end{equation}
is continuously differentiable at $t_0$. Note that, under Condition~\ref{cond:unicity:interior:t0}, \eqref{eq:theta:map} can be equivalently expressed as
\begin{equation}
  \label{eq:theta:map:2}
  \vartheta^*(t) = \arginf_{\theta \in \mathring \Theta} D_\phi(\Sc(\theta) \mid t), \qquad t \in \Nc(t_0),
\end{equation}
implying that it is then a function from $\Nc(t_0)$ to $\mathring \Theta$. As $\Sc^*$ in~\eqref{eq:S:*:map} can be expressed as
\begin{equation}
  \label{eq:S:*:vartheta}
\Sc^*(t) = \Sc(\vartheta^*(t)),\quad t \in \Nc(t_0),
\end{equation}
we see that conditions under which $\Sc^*$ is continuously differentiable at $t_0$ will follow from conditions under which $\vartheta^*$ in~\eqref{eq:theta:map:2} is continuously differentiable at~$t_0$.

The following continuity result is proven in Section~\ref{proofs:param}.

\begin{prop}[Continuity of $\vartheta^*$ at $t_0$]
  \label{prop:cont:vartheta:t0}
Under Condition~\ref{cond:unicity:t0}, the function $\vartheta^*$ in~\eqref{eq:theta:map} is continuous at $t_0$.
\end{prop}

The next corollary is an immediate consequence of the previous proposition, the continuity of $\Sc$ on $\bar \Theta$ and~\eqref{eq:S:*:vartheta}.

\begin{cor}[Continuity of $\Sc^*$ at $t_0$]
  \label{cor:cont:S:*:t0}
Under Condition~\ref{cond:unicity:t0}, the function $\Sc^*$ in~\eqref{eq:S:*:map} is continuous at $t_0$.
\end{cor}

As shall be verified in the proof of Lemma~\ref{lem:J2} (given in Section~\ref{proofs:param}), the differentiability assumptions made on $\phi$ and $\Sc$ from Section~\ref{sec:diff} onwards imply that, for any $t \in (0,1)^m$, the function $D_\phi(\Sc(\cdot) \mid t)$ is twice continuously differentiable on $\mathring \Theta$. The continuous differentiability of the function $\vartheta^*$ in~\eqref{eq:theta:map} will be shown under Condition~\ref{cond:unicity:interior:t0} and the following additional condition.

\begin{cond}[Invertibility condition]
  \label{cond:invertibility}
  Under Condition~\ref{cond:unicity:interior:t0}, the $k \times k$ matrix whose elements are
  $$
  \frac{\partial^2 D_\phi(\Sc(\theta) \mid t_0)}{\partial \theta_i \partial \theta_j} \Big|_{\theta = \vartheta^*(t_0)}, \qquad i, j \in \{1,\dots,k\},
  $$
  is positive definite.
\end{cond}

Note that the matrix appearing in the previous condition is well-defined as, under Condition~\ref{cond:unicity:interior:t0},~\eqref{eq:theta:map:2} implies that $\vartheta^*(t_0) \in \mathring \Theta$.

 \begin{remark}
  \label{rem:invertibility}
  The matrix defined in Condition~\ref{cond:invertibility} corresponds to the matrix given in Eq.~(3) of \cite{JimPinAlbMor11}. In that reference, it is claimed that, when restricted to probability vectors, Condition~\ref{cond:unicity:interior:t0} with $\Nc(t_0) = \{t_0\}$ (which corresponds to their Assumption~3) implies Condition~\ref{cond:invertibility}. However, Condition~\ref{cond:unicity:interior:t0} with $\Nc(t_0) = \{t_0\}$ is equivalent to saying that the function $D_\phi(\Sc(\cdot) \mid t_0)$ has a unique minimum at $\theta_0 = \vartheta^*(t_0) \in \mathring \Theta$ which, from second order necessary optimality conditions, only implies the positive semi-definiteness of the matrix in Condition~\ref{cond:invertibility}.

  We provide here a simple counterexample, illustrating that Condition~\ref{cond:invertibility} cannot be dispensed with. Let $\phi(x) = (x - 1)^4$, $x \in [0, \infty)$. The function $\phi$ is strictly convex on its domain and twice continuously differentiable on $(0, \infty)$. We restrict our attention to probability vectors. Let $q_0 = (1/3, 1/3, 1/3)$ be the element to be $\phi$-projected on the subset of probability vectors $\Mc = \Sc(\bar \Theta)$ where $\Theta = (0,1/2)$ and $\Sc(\theta) = (\theta, \theta, 1 - 2\theta)$, $\theta \in \bar \Theta$. The function $\Sc$ is clearly twice continuously differentiable on $\mathring \Theta = \Theta$. Let $h$ be the function from $\bar \Theta$ to $\R$ defined by
$$
h(\theta) = D_\phi( \Sc(\theta) \mid q_0), \qquad \theta \in \bar \Theta.
$$
It is then easy to verify that
$$
h(\theta) = \frac{1}{3}( 3 \theta- 1)^4 + \frac{1}{3}( 3 \theta- 1)^4 + \frac{1}{3} ( 3 (1 - 2\theta) - 1 )^4, \qquad \theta \in \bar \Theta,
$$
and that $h$ attains its unique minimum at $\theta_0 = 1/3  \in \mathring \Theta$. Standard calculations show that
$$
h''(\theta) = 36( 3 \theta- 1)^2 + 36( 3 \theta- 1)^2 + 144 [ 3 (1 - 2\theta) - 1 ]^2, \qquad \theta \in \mathring \Theta,
$$
and that $h''(\theta_0) = h''(\vartheta^*(q_0)) = 0$.  \qed
\end{remark}

Before stating one of the main results of this work, let us express the matrix defined in Condition~\ref{cond:invertibility} under the elegant form given in Eq.~(3) of \cite{JimPinAlbMor11}. To do so, additional definitions are needed to be able to properly deal with the differentiability of matrix-valued functions. Let $a,b,c,d \in \N^+$. For any function $g : U \rightarrow \R^{b \times c}$, where $U$ is an open subset of $\R^a$, we shall denote by $g_{i,j}$, $i \in \{1,\dots,b\}$, $j \in \{1,\dots, c\}$, its $b c$ component functions. Furthermore, we shall say that $g$ is $r$-times continuously differentiable on $U$, $r \geq 1$, if, for any $i \in \{1,\dots,b\}$, $j \in \{1,\dots, c\}$, the $r$-order partial derivatives of $g_{i,j}$ exist and are continuous on $U$. Next, consider the (column-major) vectorization operator which maps any function $g : \R^a \rightarrow \R^{b \times c}$ to its vectorized version $\ol{g} : \R^a \rightarrow \R^{bc}$ whose $b c$ components functions $\ol{g}_1,\dots,\ol{g}_{bc}$ are defined by $\ol{g}_{i + (j - 1) b} = g_{i,j}$, $(i,j) \in \{1,\dots,b\} \times \{1,\dots, c\}$. For any continuously differentiable function $g : U \rightarrow \R^{b \times c}$, where $U$ is an open subset of $\R^a$, we then define its Jacobian matrix at $y \in U \subset \R^a$ as the Jacobian matrix $J[{\ol{g}}](y)$ of the vectorized version $\ol{g}$ of $g$ at~$y$. In other words,
$$
J[g](y) = J[{\ol{g}}](y) =
\begin{bmatrix}
\frac{\partial \ol{g}_{1}(x)}{\partial x_1} \Big |_{x = y} & \dots & \frac{\partial \ol{g}_{1}(x)}{\partial x_a} \Big |_{x = y} \\
\vdots &  & \vdots \\
\frac{\partial \ol{g}_{b c}(x)}{\partial x_1} \Big |_{x = y} & \dots & \frac{\partial \ol{g}_{b c}(x)}{\partial x_a} \Big |_{x = y} \\
\end{bmatrix}, \qquad y \in U \subset \R^a.
$$
Note that $J[g] = J[{\ol{g}}]$ is a function from $U$ to $\R^{b c \times a}$. If $g$ is twice continuously differentiable on $U$, then, for any $y \in U$, we define $J_2[g](y)$ to be the Jacobian of the function $\ol{J[g]}$ at $y$, that is, $J_2[g](y) = J[\ol{J[g]}](y)$, $y \in U$. Some thought reveals that, with the previous definitions, it makes no difference in terms of Jacobian whether a given function is regarded as taking its values in $\R^{b \times 1}$, $\R^{1 \times b}$ or $\R^b$. In a related way, we follow the usual convention that vectors, when they appear in matrix expressions, are regarded as $1$-column matrices.

The following result, proven in Section~\ref{proofs:param}, can be used to obtain an explicit expression of the matrix appearing in Condition~\ref{cond:invertibility}.

\begin{lem}
  \label{lem:J2}
For any $t \in (0,1)^m$ and $\theta \in \mathring \Theta$, we have that
\begin{multline}
  \label{eq:J2}
  J_2[D_\phi(\Sc(\cdot) \mid t)](\theta) = \\  J[\Sc](\theta)^\top  \diag \left(\frac{1}{t_1}\phi'' \left( \frac{\Sc_1(\theta)}{t_1} \right), \dots, \frac{1}{t_m}\phi'' \left( \frac{\Sc_m(\theta)}{t_m} \right)  \right) J[\Sc](\theta) \\
+ \left( I_k \otimes \left(\phi' \left( \frac{\Sc_1(\theta)}{t_1} \right), \dots, \phi' \left( \frac{\Sc_m(\theta)}{t_m} \right) \right)^\top \right) J_2[\Sc] (\theta),
\end{multline}
where $\Sc_1, \dots, \Sc_m$ are the $m$ component functions of $\Sc$ and the symbol $\otimes$ denotes the Kronecker product.
\end{lem}

We can now state one of our main results. Its proof is given in Section~\ref{proofs:param}.

\begin{thm}[Differentiability of $\vartheta^*$ at $t_0$]
  \label{thm:diff:vartheta}
  Under Conditions \ref{cond:unicity:interior:t0} and \ref{cond:invertibility}, $\vartheta^*$ in \eqref{eq:theta:map:2} is continuously differentiable at $t_0 = (t_{0,1}, \dots, t_{0,m}) \in (0,1)^m$ with Jacobian matrix at $t_0$ given by
  \begin{equation}
    \label{eq:J:vartheta:map}
    J[\vartheta^*](t_0) = J_2[D_\phi(\Sc(\cdot) \mid t_0)](\vartheta^*(t_0))^{-1} J[\Sc](\vartheta^*(t_0))^\top \Delta(t_0) ,
  \end{equation}
  where the matrix-valued function $J_2[D_\phi(\Sc(\cdot) \mid t_0)]$ is as in~\eqref{eq:J2} with $t = t_0$ and
  \begin{multline}
    \label{eq:Delta:t0}
  \Delta(t_0) = \diag\left( \frac{\Sc_1(\vartheta^*(t_0))}{t_{0,1}^2} \phi''\left(\frac{\Sc_1(\vartheta^*(t_0))}{t_{0,1}} \right), \dots \right. \\ \left. \dots, \frac{\Sc_m(\vartheta^*(t_0))}{t_{0,m}^2} \phi''\left(\frac{\Sc_m(\vartheta^*(t_0))}{t_{0,m}} \right) \right).
  \end{multline}
\end{thm}

The next corollary is an immediate consequence of~\eqref{eq:S:*:vartheta}, Theorem~\ref{thm:diff:vartheta} and the chain rule.

\begin{cor}[Differentiability of $\Sc^*$ at $t_0$]
  \label{cor:diff:S:*}
  Assume that Conditions~\ref{cond:unicity:interior:t0} and~\ref{cond:invertibility} hold. Then, the function $\Sc^*$ in~\eqref{eq:S:*:map} is continuously differentiable at $t_0 \in (0,1)^m$ with Jacobian matrix at $t_0$ given by
  \begin{equation}
    \label{eq:J:S:*:map}
    J[\Sc^*](t_0) = J[\Sc](\vartheta^*(t_0))  J[\vartheta^*](t_0),
  \end{equation}
where $J[\vartheta^*](t_0)$ is given in~\eqref{eq:J:vartheta:map}.
\end{cor}

\begin{remark}
  \label{rem:conds}
  When $t_0$ is known, one can attempt to verify the conditions of Theorem~\ref{thm:diff:vartheta} or Corollary~\ref{cor:diff:S:*} analytically, or at least empirically. For Condition~\ref{cond:unicity:interior:t0}, one could choose a few vectors $t$ in a ``neighborhood'' of $t_0$ and attempt to verify (at least numerically) that the function $D_\phi(\Sc(\cdot) \mid t)$ has a unique minimum on $\mathring \Theta \subset \R^k$. When $k \in \{1,2\}$ in particular, the preceding verification could be graphical (see Section~\ref{sec:binom}). As to Condition~\ref{cond:invertibility}, one could simply compute the matrix $J_2[D_\phi(\Sc(\cdot) \mid t_0)](\vartheta^*(t_0))$ and attempt to invert it. In a statistical context, $t_0$ is typically unknown, though. Naturally, one can simply replace $t_0$ by a consistent estimator $t_n$ in the previous strategies. However, when the set $\Mc$ is convex, there is typically no need for such approximate checks as Conditions~\ref{cond:unicity:interior:t0} and~\ref{cond:invertibility} are automatically verified for a rather large class of \phidivergences\, and many situations of practical interest, as exposed in Section~\ref{sec:convex}. \qed
\end{remark}

\begin{remark}
  %\label{rem:continuity:Jacobian}
  Under Conditions~\ref{cond:unicity:interior:t0} and~\ref{cond:invertibility}, Theorem~\ref{thm:diff:vartheta} and Corollary~\ref{cor:diff:S:*} imply that the functions $J[\vartheta^*]$ and $J[\Sc^*]$ are continuous at $t_0$. In statistical applications, given a consistent estimator $t_n$ of (the unknown) $t_0$, we thus immediately obtain from the continuous mapping theorem that $J[\vartheta^*](t_0)$ (resp.\ $J[\Sc^*](t_0)$) is consistently estimated by $J[\vartheta^*](t_n)$ (resp.\ $J[\Sc^*](t_n)$). \qed
\end{remark}

We end this subsection by stating two simple conditions which imply Condition~\ref{cond:unicity:interior:t0} when combined with Condition~\ref{cond:unicity:t0}.

\begin{cond}[Interior]
  \label{cond:interior}
  The function $\Sc$ in Condition~\ref{cond:M} satisfies $\Sc(\mathring \Theta) = \Mc \cap (0,1)^m$.
\end{cond}

\begin{cond}[Same support] %[Same support condition]
  \label{cond:support}
   For any $t \in (0,1)^m$, all $\phi$-projections of~$t$ on $\Mc = \Sc(\bar \Theta)$ belong to $(0,1)^m$.
\end{cond}

Condition~\ref{cond:interior} is satisfied in many applications (see \change{Section~\ref{sec:app:asym}}) while the verification of Condition~\ref{cond:support} is discussed in the forthcoming subsection. The proof of the next result is given in Section~\ref{proofs:param}.

\begin{lem}
  \label{lem:cond:unicity:interior}
  Conditions~\ref{cond:unicity:t0}, \ref{cond:interior} and~\ref{cond:support} imply Condition~\ref{cond:unicity:interior:t0}.
\end{lem}

\subsection{The case of $\Mc$ convex}
\label{sec:convex}

Recall that Conditions~\ref{cond:phi:strict:convex},~\ref{cond:phi:diff} and~\ref{cond:M} are supposed to hold. Below, using results of \cite{GieRef17}, we first prove (see Section~\ref{proofs:convex}) that Condition~\ref{cond:unicity:t0} is automatically satisfied under an additional assumption of convexity of the set $\Mc$ on which we wish to $\phi$-project the vector $t_0 \in (0,1)^m$.

\begin{prop}
  \label{prop:cond:unicity:t0}
  Assume that $\Mc = \Sc(\bar \Theta) \subset [0,1]^m$ is convex. Then, for any $t \in (0,1)^m$, the $\phi$-projection $s^*$ of $t$ on $\Mc = \Sc(\bar \Theta)$ (exists and) is unique, which implies that Condition~\ref{cond:unicity:t0} holds.
\end{prop}

Since Condition~\ref{cond:unicity:t0} holds when $\Mc$ is convex, from Lemma~\ref{lem:cond:unicity:interior}, it suffices to verify Conditions~\ref{cond:interior} and~\ref{cond:support} to obtain Condition~\ref{cond:unicity:interior:t0}. 

\begin{remark}
  \label{rem:KL:support}
  Restricting attention to probability vectors, Condition~\ref{cond:support} holds for the Kullback--Leibler divergence (that is, for $I$-projections)  since $\Mc = \Sc(\bar \Theta)$ contains probability vectors in $(0,1)^m$ by construction, as $\Sc(\mathring \Theta) \subset \Mc \cap (0,1)^m$ and $\Sc(\mathring \Theta) \neq \emptyset$ by Condition~\ref{cond:M}. Indeed, when one wishes to $I$-project a probability vector $q_0$ on a closed convex subset containing at least one probability vector with the same support as $q_0$, it is known from \citet[Theorem~2.1, Theorem~2.2 and the subsequent remark]{Csi75} that the support of the unique $I$-projection will be equal to that of $q_0$.  \qed 
\end{remark}

Let us now turn to Condition~\ref{cond:invertibility}, the other key condition in Theorem~\ref{thm:diff:vartheta}. To verify it when $t_0$ is unknown (see Remark~\ref{rem:conds}), one could rely on more general conditions that imply Condition~\ref{cond:invertibility}. One such condition is as follows.

\begin{cond}[Strong convexity of $D_\phi(\Sc(\cdot) \mid t)$]
  \label{cond:S:strong}
  The set $\bar \Theta$ is convex and, for any $t \in (0,1)^m$, the function $D_\phi(\Sc(\cdot) \mid t)$ from $\bar \Theta$ to $\R$ is strongly convex.
\end{cond}

The next proposition is proven in Section~\ref{proofs:convex}.

\begin{prop}
  \label{prop:S:strong:invert}
  Condition~\ref{cond:S:strong} implies that, for any $t \in (0,1)^m$ and $\theta \in \mathring \Theta$, the $k \times k$ matrix $J_2[D_\phi(\Sc(\cdot) \mid t)](\theta)$ is positive definite, and thus that Condition~\ref{cond:invertibility} holds. 
\end{prop}

We have (from Proposition~\ref{prop:D:phi:strong:convex} and Table~\ref{tab:phi}) that, for many \phidivergences\, and any $t \in (0,1)^m$, the function $D_\phi(\cdot \mid t)$ is strongly convex on $[0,1]^m$. Thus, Proposition~\ref{prop:S:strong:invert} suggests to consider conditions on $\Sc$ and $\Theta$ under which Condition~\ref{cond:S:strong} holds. A simple such condition is as follows.

\begin{cond}[$\Sc$ is affine]
  \label{cond:S:affine}
  The set $\bar \Theta$ is convex and the function $\Sc$ from $\bar \Theta$ to $[0,1]^m$ is affine, that is, there exists an $m \times k$ matrix $A$ and $\gamma \in \R^m$ such that $\Sc(\theta) = A \theta + \gamma$, $\theta \in \bar \Theta$.
\end{cond}

The following lemma is proven in Section~\ref{proofs:convex}.

\begin{lem}[Strong convexity of $D_\phi(\Sc(\cdot) \mid t)$]
  \label{lem:cond:S}
Conditions~\ref{cond:phi:strong:convex} and~\ref{cond:S:affine} imply that Condition~\ref{cond:S:strong} holds.
\end{lem}

Note that, from Lemma~\ref{lem:phi:strict:convex}, Condition~\ref{cond:phi:strong:convex} can be viewed as a strengthening of Condition~\ref{cond:phi:strict:convex}. The previous derivations lead to the following result, proven in Section~\ref{proofs:convex}.

\begin{cor}[Differentiability of $\vartheta^*$ and $\Sc^*$ at $t_0$]
  \label{cor:diff:convex}
  Assume that $\Mc = \Sc(\bar \Theta)$ is convex and that Conditions~\ref{cond:phi:strong:convex},~\ref{cond:interior},~\ref{cond:support} and~\ref{cond:S:affine} hold. Then, the function $\vartheta^*$ in~\eqref{eq:theta:map:2} is continuously differentiable at $t_0 = (t_{0,1},\dots,t_{0,m}) \in (0,1)^m$ with Jacobian matrix at $t_0$ given by
  \begin{multline}
    \label{eq:J:vartheta:map:affine}
    J[\vartheta^*](t_0) = \left[ A^\top \diag \left(\frac{1}{t_{0,1}}\phi'' \left( \frac{\Sc_1(\vartheta^*(t_0))}{t_{0,1}} \right), \dots \right. \right. \\ \left. \left. \dots, \frac{1}{t_{0,m}}\phi'' \left( \frac{\Sc_m(\vartheta^*(t_0))}{t_{0,m}} \right)  \right)   A \right]^{-1} A^\top \Delta(t_0),
\end{multline}
where $\Delta(t_0)$ is defined in~\eqref{eq:Delta:t0}, and the function $\Sc^*$ in~\eqref{eq:S:*:map} is continuously differentiable at $t_0$ with Jacobian matrix at $t_0$ given by $J[\Sc^*](t_0) = A  J[\vartheta^*](t_0)$.
\end{cor}

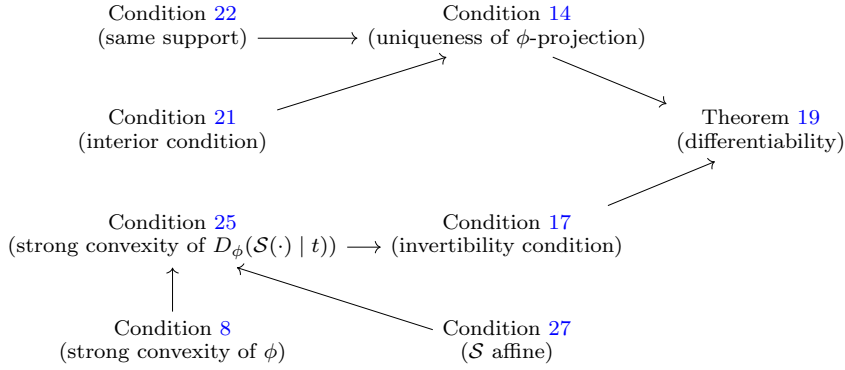
\begin{figure}[t!]
  \centering
  \begin{equation*}
    \begin{tikzcd}[row sep=2em, column sep=0.5em]% Define nodes
      |[alias=condsupport]|  \shortstack[c]{ \text{Condition \ref{cond:support}} \\ (same support)}  & |[alias=conduniqproj]| \shortstack{\text{Condition \ref{cond:unicity:interior:t0}}  \\ (unicity of $\phi$-projection)} \\
      |[alias=condint]| \shortstack{\text{Condition \ref{cond:interior}} \\ (interior condition)} & & |[alias=thmdiff]| \shortstack{\text{Theorem \ref{thm:diff:vartheta}} \\ (differentiability)}  \\
      |[alias=condSstrong]|  \shortstack{\text{Condition \ref{cond:S:strong}} \\ (strong convexity of $D_\phi(\Sc(\cdot) \mid t)$)} & |[alias=condinvert]| \shortstack{\text{Condition \ref{cond:invertibility}} \\ (invertibility condition)}\\
      |[alias=condphistrongconvex]|  \shortstack{\text{Condition \ref{cond:phi:strong:convex}} \\ (strong convexity of $\phi$)} &  |[alias=condSaffine]|  \shortstack{\text{Condition \ref{cond:S:affine}} \\ ($\Sc$ affine)} 
      % 
      % Define arrows
      \arrow[from = conduniqproj, to=thmdiff]
      \arrow[from = condinvert, to=thmdiff]
      \arrow[from = condint, to=conduniqproj]
      \arrow[from= condsupport, to = conduniqproj ]
      \arrow[from= condSstrong, to = condinvert ]
      \arrow[from= condphistrongconvex, to = condSstrong ]
      \arrow[from= condSaffine, to = condSstrong ]
    \end{tikzcd}
  \end{equation*}
  \caption{Diagram summarizing the various implications of conditions for differentiability when the set $\Mc$ of interest is convex.}
  \label{fig:diagram}
\end{figure}

Figure~\ref{fig:diagram} summarizes the various implications of conditions leading to Corollary~\ref{cor:diff:convex}. The only condition in that result that may not be easily verifiable is Condition~\ref{cond:support}. Its verification would obviously be immediate if $t_0$ was known and its $\phi$-projection on $\Mc$ could be computed. However, as discussed in Remark~\ref{rem:conds}, this is usually not the case in a statistical context. From Remark~\ref{rem:KL:support}, though, we know that Condition~\ref{cond:support} is automatically verified for the Kullback--Leibler divergence when we restrict attention to probability vectors. Using results of \cite{Rus87}, we can prove (see Section~\ref{proofs:convex}) that this may also be the case for other \phidivergences\, provided that Condition~\ref{cond:S:affine} holds and $\Theta$ is defined by linear inequalities. 
%   \qed
% \end{remark}

\begin{prop}
  \label{prop:cond:support}
  Assume that Condition~\ref{cond:S:affine} holds, $\lim_{x \to 0^+} \phi'(x) = - \infty$, $\Theta \subset \R^k$ is defined by linear inequalities and $\Mc = \Sc(\bar \Theta) \subset [0,1]^m$ is a convex subset of probability vectors. Then, for any probability vector $q \in (0,1)^m$, the (unique) $\phi$-projection $p^*$ of $q$ on $\Mc = \Sc(\bar \Theta)$ belongs to $(0,1)^m$.
\end{prop}

\begin{remark}
  \label{rem:support}
  For the \phidivergences\, listed in Table~\ref{tab:phi}, $\lim_{x \to 0^+} \phi'(x) = - \infty$ holds for the Kullback--Leibler divergence, the squared Hellinger divergence, the reverse relative entropy, the Jensen--Shannon divergence, Neyman's $\chi^2$ divergence and $\alpha$-divergences with $\alpha < 1$.
  \qed
\end{remark}

The previous proposition leads to the following corollary, of importance for a large class of applications. It is proven in Section~\ref{proofs:convex}.

\begin{cor}
  \label{cor:diff:linear:ineq}
  Assume that Conditions~\ref{cond:phi:strong:convex},~\ref{cond:interior} and~\ref{cond:S:affine} hold, $\lim_{x \to 0^+} \phi'(x) = - \infty$, $\Theta$ is defined by linear inequalities and $\Mc = \Sc(\bar \Theta) \subset [0,1]^m$ is a convex subset of probability vectors. Then, for any probability vector $q \in (0,1)^m$, the function $\vartheta^*$ in~\eqref{eq:theta:map:2} is continuously differentiable at $q$ with Jacobian matrix at $q$ given by~\eqref{eq:J:vartheta:map:affine} with $t_0 = q$, and the function $\Sc^*$ in~\eqref{eq:S:*:map} is continuously differentiable at $q$ with Jacobian matrix at $q$ given by $J[\Sc^*](q) = A  J[\vartheta^*](q)$.
\end{cor}

\subsection{The case of $\Mc$ defined by linear equalities}
\label{sec:lin:eq}

Many applications involve $\phi$-projections on sets defined by linear equalities (see Sections~\ref{sec:moments} and~\ref{sec:Frechet}). The proposition below, which states that in such cases Conditions~\ref{cond:M},~\ref{cond:interior} and~\ref{cond:S:affine} are automatically verified, is therefore of great interest. It is proven in Section~\ref{proofs:convex}.

\begin{prop}
  \label{prop:linear:eq}
  Assume that the set $\Mc$ is defined by linear equalities, that is, that there exist $d \in \N^+$, $\alpha = (\alpha_1, \dots, \alpha_d) \in \R^d$ and $d$ functions $f_1, f_2, \dots, f_d$ on $\{1,\dots,m\}$ such that
  $$
  \Mc = \left\{ s \in [0,1]^m : \sum_{i = 1}^m s_i f_j(i) = \alpha_j, j \in \{1, \dots, d\} \right\}.
  $$
  Moreover, suppose that $\Mc \cap (0,1)^m$ is nonempty and $\Mc$ is not reduced to a singleton. Then Conditions~\ref{cond:M},~\ref{cond:interior} and~\ref{cond:S:affine} hold.    
\end{prop}

Note that the previous result does not directly provide the expression of the function $\Sc$ in Condition~\ref{cond:M} which is necessary to compute the Jacobians in Corollaries~\ref{cor:diff:convex} and~\ref{cor:diff:linear:ineq}. It merely guarantees that, whenever $\Mc$ is defined by linear equalities, the framework for studying the differentiability of $\phi$-projections set up in Section~\ref{sec:framework} applies with $\Sc$ affine and Condition~\ref{cond:interior} additionally satisfied.

\section{Applications to minimum \phidivergence\ estimators}
\label{sec:app}

This section presents two applications of the results of Section~\ref{sec:diff} to $\phi$-projection estimators. To indicate that the element of $(0,1)^m$ to be $\phi$-projected is a probability vector, it will be denoted by $q_0$. Given an estimator $q_n$ of $q_0$, we call $\Sc^*(q_n)$ the $\phi$-projection and/or minimum \phidivergence\ estimator of $q_0$ (that is, the estimator of $\Sc^*(q_0)$). \change{The first subsection provides an immediate non-asymptotic application of our differentiability results to robust statistics. The second subsection shows how the results of Section~\ref{sec:diff} can be employed to derive the limiting distributions of properly scaled $\phi$-projection estimators in three different situations.}

\change{
\subsection{Influence functions of $\phi$-projection estimators}

A particularly useful concept in the context of robust statistics is the \emph{influence function} \citep[see, e.g.,][]{Ham74} which quantifies both the first-order robustness and the asymptotic efficiency of an estimator. Minimum \phidivergence\ estimators are also popular for their strong robustness properties \citep[see, e.g.,][]{BasShiPar11}. Although the use of influence functions for these estimators has been questioned \citep{Lin94}, their computation may still be of interest. For example, in the continuous case, as a starting point for the investigation of robustness under pre-specified restrictions of the parameter space (such as those imposed by a composite null hypothesis), \citet[Theorem 2.1]{Gho15} derived a general expression for the influence functions of density-based minimum \phidivergence\ estimators. In this subsection, we demonstrate how an analog result can be obtained for minimum \phidivergence\ estimators in the finite discrete case as a corollary of Theorem~\ref{thm:diff:vartheta}. 

In what follows, we adapt the classical definitions employed in robust statistics \citep[see, e.g.,][]{Lin94, Gho15} to our setting. In accordance with, for example, Definition 2.2 in \cite{Gho15}, the \emph{statistical functional} associated with the minimum \phidivergence\ estimator of a probability vector $q_0 \in (0,1)^m$ is $\vartheta^*$ defined in \eqref{eq:theta:map}. For any $j \in \{1,\dots,m\}$, let $q_0^{[j]}: (-1, 1) \to (0,1)^m$ be the function defined by 
\begin{equation}
  \label{eq:tj}
q_0^{[j]}(\epsilon) = (1 - \epsilon) q_0 + \epsilon e_j, \qquad \epsilon \in (-1, 1),
\end{equation}
where $e_j$ is the vector of $[0,1]^m$ with a 1 at position $j$ and 0 elsewhere. This allows us to define, for any $\epsilon \in [0, 1)$ and $j \in \{1,\dots,m\}$, the $\epsilon$-contaminated at $j$ version of $q_0$ as $q_0^{[j]}(\epsilon)$. The influence function of the minimum \phidivergence\ estimator of $q_0$ is then
\begin{equation}
  \label{eq:infl}
	\textup{IF}(j, \vartheta^*, q_0) = \frac{\partial \vartheta^*(q^{[j]}_{0}(\epsilon))}{\partial \epsilon} \Big |_{\epsilon = 0}, \qquad j \in \{1,\dots,m\}.
\end{equation}
%For $t \in (0,1)^m$ and $j \in \{1,\dots,m\}$, by introducing the map $\vartheta_t^{*,[j]} : [0, \infty) \to \Theta$ defined as 
%
%\begin{equation}\label{infl:theta:eps}
%\epsilon \mapsto \vartheta^*(t^{[j]}(\epsilon)), \quad \epsilon \in [0, \infty),
%\end{equation}
%we may observe that the computation of the influence function \eqref{infl} in our context consists in computing the Jacobian matrix, \change{whenever well defined}, $J[\vartheta_{q_0}^{*,[j]}]$ (which amounts to a vector in this case), of $\vartheta_{q_0}^{*,[j]}$. That is, we may write
%
%\begin{equation}\label{infl:jac}
%IF(j, \vartheta^*, q_0) = J[\vartheta_{q_0}^{*,[j]}](0), \quad j \in \{1,\dots,m\},
%\end{equation}
%with a small abuse of notation.
%that is, the directional derivative of $\vartheta^*$ at $q_0$ in the direction $e_j - q_0$ \citep[see, e.g.,][p 366, for a definition for real valued functions and consider the componentwise extension for vector valued functions]{Fit09}.

The following result, proven in Section~\ref{proofs:IF}, is a straightforward consequence of Theorem~\ref{thm:diff:vartheta} and the chain rule.

%\begin{prop}\label{prop:IF}
%\change{Assume Conditions \ref{cond:unicity:interior:t0} and \ref{cond:invertibility} hold for $q_0$. We} have that
% 
%\begin{equation}\label{eq:infl:theta:eps:jac}
%J[\vartheta_{q_0}^{*,[j]}](0) = J[\vartheta^*](q_0) (e_j - q_0), \quad j \in \{1,\dots,m\},
%\end{equation}
%with $\vartheta^*$ defined in \eqref{eq:theta:map}, 
%$\vartheta_{q_0}^{*,[j]}$ as in \eqref{infl:theta:eps} for $t = q_0$, $J[\vartheta^*]$ given in \eqref{eq:J:vartheta:map} and $e_j$ the element in $[0,1]^m$ with a 1 in the $j$-th component and 0 elsewhere. Then, according to \eqref{infl:jac}, we can write
%
%$$
%IF(j, \vartheta^*, q_0) = J[\vartheta^*](q_0) (e_j - q_0), \quad j \in \{1,\dots,m\}.
%$$
%\end{prop}

\begin{cor}
  \label{cor:IF}
Assume that Conditions~\ref{cond:unicity:interior:t0} and~\ref{cond:invertibility} hold with $t_0 = q_0$. Then,  
$$
\textup{IF}(j, \vartheta^*, q_0) = J[\vartheta^*](q_0) (e_j - q_0), \qquad j \in \{1,\dots,m\},
$$
where $\vartheta^*$ is defined in~\eqref{eq:theta:map:2} and $J[\vartheta^*](q_0)$ is given in \eqref{eq:J:vartheta:map} with $t_0 = q_0$. 
\end{cor}

}
%Then, given \eqref{eq:infl}, the result follows from an application of the directional derivative formula for real valued functions \citep[see, for example,][Theorem 13.16 and (13.24)]{Fit09} applied componentwise to compute the directional derivative of $\vartheta^*$ at $q_0$ in the direction $e_j - q_0$.

%Additionally, if $q_0 \in \mathcal{M}$, that is, there exist $\theta_0$ such that $q_0 = \mathcal{S}(\theta_0)$, then
%$$
%F(j, \theta^*, \mathcal)
%$$

\subsection{Asymptotics of minimum \phidivergence\ estimators}
\label{sec:app:asym}

In the following \change{three examples}, after defining the function $\Sc$ in Condition~\ref{cond:M} and verifying the conditions of Corollary~\ref{cor:diff:S:*}, Corollary~\ref{cor:diff:convex} or Corollary~\ref{cor:diff:linear:ineq}, we shall compute the Jacobian matrix of the function $\Sc^*$ in~\eqref{eq:S:*:map} at $q_0$, that is, in the notation of Section~\ref{sec:diff}, $J[\Sc^*](q_0)$. \change{Recall that $q_n$ is an estimator of $q_0$.} If $\sqrt{n} (q_n - q_0)$ converges weakly as $n$ tends to infinity, it follows immediately with the delta method \citep[see, e.g.,][Theorem~3.1]{Van98} that
\begin{equation}
  \label{eq:delta}
  \sqrt{n} (\Sc^*(q_n) - \Sc^*(q_0)) = J[\Sc^*](q_0) \sqrt{n} (q_n - q_0) + o_\Pr(1).
\end{equation}

For illustration purposes, we shall specifically consider a random variable $X$ taking its values in a set $\Xc = \{x_1,\dots,x_m\}$ of interest such that $\Pr(\{X = x_i\}) = q_{0,i}$, $i \in \{1,\dots,m\}$, and we shall assume that we have at hand $n$ independent copies $X_1,\dots,X_n$ of $X$. A natural estimator of $q_0$ is then the probability vector $q_n$ whose components are
\begin{equation*}
 q_{n,i} = \frac{1}{n} \sum_{j=1}^n \1(X_j = x_i), \qquad i \in \{1,\dots,m\},
\end{equation*}
where $\1$ denotes the indicator function. The multivariate central limit theorem establishes that, as $n$ tends to infinity, $\sqrt{n} (q_n - q_0)$ converges weakly to an $m$-dimensional centered normal distribution with covariance matrix $\Sigma_{q_0} = \diag(q_0) - q_0 q_0^\top$. As a result, we obtain from~\eqref{eq:delta} that, as $n$ tends to infinity, $\sqrt{n} (\Sc^*(q_n) - \Sc^*(q_0))$ converges weakly to an $m$-dimensional centered normal distribution with covariance matrix $\Sigma = J[\Sc^*](q_0) \Sigma_{q_0} J[\Sc^*](q_0)^\top$. To empirically verify the correctness of the asymptotic covariance matrix $\Sigma$, in all three applications, we shall finally compare it with $\Sigma_{n,N}$, the sample covariance matrix computed from $N=5000$ independent replicates of $\sqrt{n} \Sc^*(q_n)$ for $n = 5000$.
The first \change{example} below is prototypical of the use of minimum \phidivergence\ estimators for parametric inference: it is about $\phi$-projecting a probability vector $q_0$ on the set $\Mc$ of probability vectors generated from binomial distributions whose first parameter is fixed to $m-1$. As the set $\Mc$ is not convex in this case, we will have to rely on the general results of Section~\ref{sec:param}, and on Corollary~\ref{cor:diff:S:*} in particular. In the second and third scenarios, the set $\Mc$ is defined by linear equalities (and is thus convex) which allows using the results of Sections~\ref{sec:convex} and~\ref{sec:lin:eq}. Specifically, in the second (resp.\ third) \change{example}, $\Mc$ is the set of probability vectors generated from distributions with certain moments fixed (resp.\ the Fréchet class of all bivariate probability arrays with given univariate margins).

\change{\subsubsection{$\phi$-projection on the set of binomial probability vectors}}
\label{sec:binom}

We first consider a parametric inference scenario in which one wishes to $\phi$-project a probability vector $q_0 \in (0,1)^m$ of interest on a given parametric model. As an example, for $m > 1$, we consider the set of probability vectors generated from binomial distributions whose first parameter is fixed to $m-1$. Specifically, for any $\ell \in \N_0$ and $i \in \{0,\dots,\ell\}$, let
$$
p_{\ell,i}(\theta) = \binom{\ell}{i} \theta^{i} (1 - \theta)^{\ell-i}, \qquad \theta \in [0,1],
$$
and let $\Mc = \{(p_{m-1,0}(\theta),\dots,p_{m-1,m-1}(\theta)) : \theta \in [0,1] \}$. The set $\Theta$ in Condition~\ref{cond:M} is thus $[0,1]$ and the $i$th component function $\Sc_i$, $i \in \{1,\dots,m\}$, of the function $\Sc$ is $p_{m-1,i-1}$. Furthermore, for any $\ell \in \N_0$, $i \in \{0,\dots,\ell\}$ and $\theta \in (0,1)$, 
\begin{align*}
  p_{\ell,i}'(\theta) =& \binom{\ell}{i} \left[\1(i > 0) i \theta^{i-1} (1 - \theta)^{\ell-i} - \1(i < \ell) \theta^{i} (\ell-i) (1 - \theta)^{\ell-i-1}\right] \\
  =& \ell \left[ \1(i \geq 1) p_{\ell-1,i-1}(\theta) - \1(i \leq \ell-1) p_{\ell-1,i}(\theta) \right],
\end{align*}
and 
\begin{align*}
  p_{\ell,i}''(\theta) =  \ell \left[ \1(i \geq 1) p_{\ell-1,i-1}'(\theta) - \1(i \leq \ell-1) p_{\ell-1,i}'(\theta) \right].
\end{align*}
From the above, we obtain that, for any $\theta \in (0,1)$,
$$
J[\Sc](\theta) = (p_{m-1,0}'(\theta),\dots,p_{m-1,m-1}'(\theta)) %\quad \text{ and } \quad
$$
and
$$
J_2[\Sc](\theta) = (p_{m-1,0}''(\theta),\dots,p_{m-1,m-1}''(\theta)).
$$

\begin{figure}[t!]
\begin{subfigure}{.33\textwidth}
\centering
\includegraphics*[width=1\linewidth]{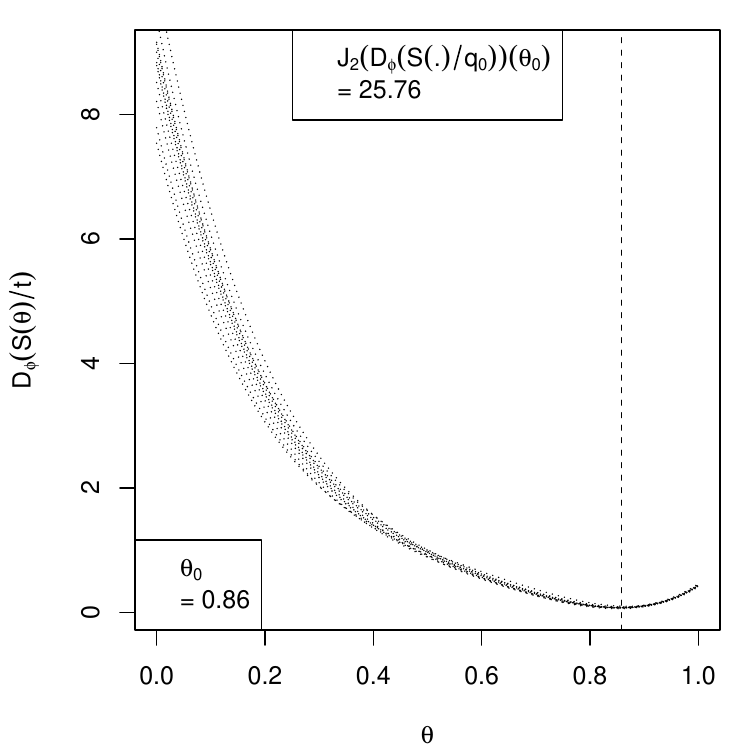}
\end{subfigure}%
\begin{subfigure}{.33\textwidth}
\centering
\includegraphics*[width=1\linewidth]{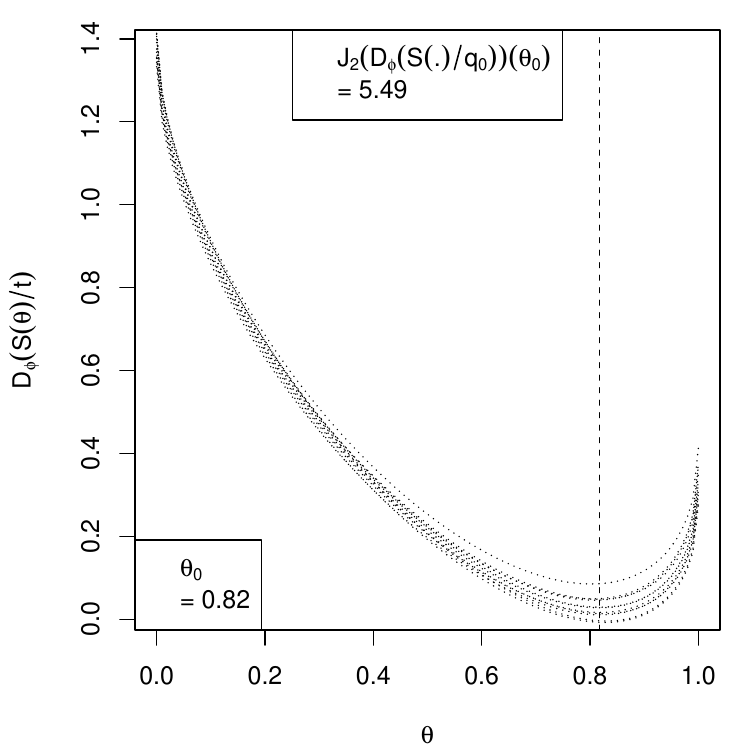}
\end{subfigure}
\begin{subfigure}{.33\textwidth}
\centering
\includegraphics*[width=1\linewidth]{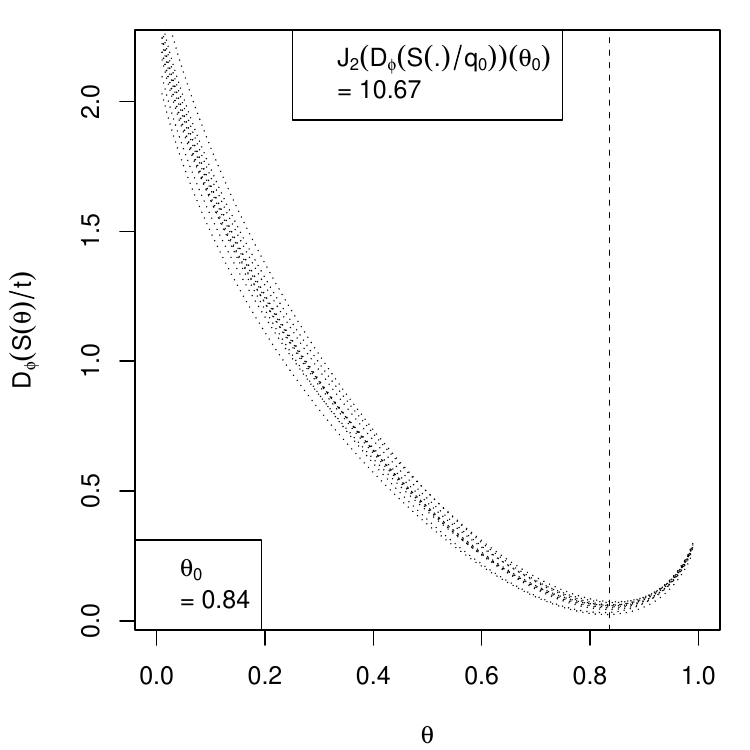}
\end{subfigure}
\caption{For Pearson's $\chi^2$ divergence (left), the squared Hellinger divergence (middle) and the Kullback--Leibler divergence (right), graphs of the functions $D_\phi(\Sc(\cdot) \mid t)$ for 10 vectors $t \in (0,1)^3$ of the form $t = q_0 + (z_1,\dots,z_m)$, where $(z_1,\dots,z_m)$ is drawn from the $m$-dimensional centered normal distribution with covariance matrix $0.01^2 I_m$. In each plot, the vertical dashed line marks the value $\theta_0$ at which the function $D_\phi(\Sc(\cdot) \mid q_0)$ reaches its minimum and the top insert contains the (strictly positive) value of $J_2[D_\phi(\Sc(\cdot) \mid q_0)](\theta_0)$.}
\label{fig:check:cond}
\end{figure}

As an illustration, we take $m=3$ and $q_0 = (0.1, 0.2, 0.7) \not \in \Mc = \Sc(\bar \Theta)$. As \phidivergence, we consider Pearson's $\chi^2$ divergence, the squared Hellinger divergence and the Kullback--Leibler divergence (see Table~\ref{tab:phi}). To be able to apply Corollary~\ref{cor:diff:S:*} and compute $J[\Sc^*](q_0)$ given in~\eqref{eq:J:S:*:map} (to eventually compute $\Sigma = J[\Sc^*](q_0) \Sigma_{q_0} J[\Sc^*](q_0)^\top$), we need to verify Conditions~\ref{cond:unicity:interior:t0} and~\ref{cond:invertibility}. We proceed as suggested in Remark~\ref{rem:conds}. To empirically verify Condition~\ref{cond:unicity:interior:t0}, we plot the function $D_\phi(\Sc(\cdot) \mid t)$ for 10 vectors $t \in (0,1)^3$ in a ``neighborhood'' of~$q_0$. Specifically, we consider vectors $t$ of the form $t = q_0 + (z_1,\dots,z_m)$, where $(z_1,\dots,z_m)$ is drawn from the $m$-dimensional centered normal distribution with covariance matrix $0.01^2 I_m$. Figure~\ref{fig:check:cond} reveals that Condition~\ref{cond:unicity:interior:t0} seems to hold for all three divergences. In all three cases, the value $\theta_0$ at which $D_\phi(\Sc(\cdot) \mid q_0)$ reaches its minimum is found numerically using the \texttt{optim()} function of the \textsf{R} statistical environment \citep{Rsystem} and its rounded value is reported in the lower left corner of each plot. The rounded value of $J_2[D_\phi(\Sc(\cdot) \mid q_0)](\theta_0)$ is also provided in each plot of Figure~\ref{fig:check:cond} and confirms that Condition~\ref{cond:invertibility} is likely to hold as well. For all three \phidivergences, the asymptotic covariance matrix $\Sigma$ of $\sqrt{n} (\Sc^*(q_n) - \Sc^*(q_0))$ and its approximation $\Sigma_{n,N}$ show very good agreement. We print their rounded versions below for Pearson's $\chi^2$ divergence:

$$
\input{binom-chi-square.tex}.
$$

\change{\subsubsection{$\phi$-projection on the set of probability vectors generated from distributions with certain moments fixed}\label{sec:moments}}

Let $X$ be a random variable whose support is $\Xc = \{x_1, x_2, \dots, x_m\}$ with $x_1 < \dots < x_m$ and let $q_0 \in (0,1)^m$ be defined by $q_{0,i} = \Pr(\{X = x_i\})$, $i \in \{1,\dots,m\}$. \change{Without loss of generality, assume} that the aim is to $\phi$-project $q_0$ on the set $\Mc$ of probability vectors obtained from distributions on $\Xc$ whose $r < m-1$ first raw (non-centered) moments are fixed to some known values $\mu_1,\dots,\mu_r$, that is,
\begin{equation}
  \label{eq:M:moments}
  \Mc = \left\{ p \in [0,1]^m : \sum_{i = 1}^m p_i = 1, \sum_{i = 1}^m x_i^j p_i = \mu_j, j \in \{ 1,\dots, r \} \right\}.
\end{equation}
For the previous problem to be well-defined, it is further necessary to assume that $\Mc$ is nonempty. Note that the above setting is strongly related to the so-called \emph{moment problem}. In the current finite discrete setting, \cite{Tag00,Tag01} proposed to address it using the maximum entropy principle \citep{Jay57}. It appears that his approach is a particular case of the considered $\phi$-projection approach when $D_\phi$ in~\eqref{eq:phi:div} is the Kullback--Leibler divergence and $q_0$ is the uniform distribution on $\Xc$.

The set $\Mc$ in~\eqref{eq:M:moments} is clearly convex. To illustrate the usefulness of the differentiability results stated in Section~\ref{sec:convex}, we shall further assume that $\Mc \cap (0,1)^m$ is nonempty. Since $\Mc$ is defined by linear equalities, Proposition~\ref{prop:linear:eq} guarantees that Conditions~\ref{cond:M} and~\ref{cond:interior} hold for some set $\Theta$ and some affine function $\Sc$ to be determined. Moreover, let $\mu = (1, \mu_1, \dots, \mu_r) \in \R^{r+1}$ and, for any $\ell \in \{1,\dots,m\}$, let $U_\ell$ be the $\ell \times m$ matrix whose $i$th row is $(x_1^{i-1}, \dots, x_m^{i-1})$, $i \in \{1,\dots,\ell\}$. With the above notation, $\Mc = \{p \in [0,1]^m : U_{r+1} p = \mu \}$. Next, take an element $p$ of $\Mc$ and denote by $\theta_1,\dots,\theta_{m-r-1}$ the corresponding raw (non-centered) moments of order $r+1,\dots,m-1$, respectively. Then, 
$$
U_m p = (\mu, \theta),
$$
where $\theta = (\theta_1,\dots,\theta_{m-r-1}) \in \R^{m-r-1}$ and
$$
(\mu, \theta) = (1, \mu_1, \dots, \mu_r, \theta_1,\dots,\theta_{m-r-1}) \in \R^m.
$$
Since $U_m$ is a Vandermonde matrix with $x_i \neq x_j$ for $i, j \in \{ 1, \dots, m \}$, $i \neq j$, it is invertible \cite[see, e.g.,][p 203]{GolLoa13} and we have
$$
p = U_m^{-1} (\mu, \theta)
$$
confirming that any p.m.f.\ on $\Xc$ is perfectly defined by its $m-1$ raw moments. As a consequence, any $p \in \Mc$ can be parametrized by a vector $\theta \in \R^{m-r-1}$ (of raw moments) in
$$
\Theta = \left \{ \theta \in \R^{m - r -1} : 0_{\R^m} \leq U_m^{-1} (\mu, \theta) \leq 1_{\R^m}  \right\}
$$
and the function $\Sc$ from $\Theta = \bar \Theta$ to $[0,1]^m$ such that $\Mc = \Sc(\bar \Theta)$ is given by
$$
\Sc(\theta) = U_m^{-1}(\mu, \theta), \qquad \theta \in \bar \Theta.
$$
Some thought reveals that $\Sc(\mathring \Theta) = \Mc \cap (0,1)^m$, that is, Condition~\ref{cond:interior} holds as expected.

Let $V_m$ (resp. $W_m$) be the $m \times (r+1)$ (resp.\ $m \times (m - r - 1)$) matrix obtained from $U_m^{-1}$ by keeping its first $r+1$ (resp.\ last $m - r - 1$) columns. Then, note that $\Theta$ and $\Sc$ can be equivalently expressed as
\begin{equation}
  \label{eq:Theta:moments}
  \Theta = \left \{ \theta \in \R^{m - r -1} : - V_m \mu \leq W_m \theta \leq 1_{\R^m} - V_m \mu  \right\}
\end{equation}
and
\begin{equation}
  \label{eq:affine:moments}
\Sc(\theta) = W_m \theta + V_m \mu , \qquad \theta \in \bar \Theta,
\end{equation}
respectively. From the previous expressions, we have that $\bar \Theta$ is convex, $\Sc$ is affine (that is, Condition~\ref{cond:S:affine} holds) and the Jacobian matrix of $\Sc$ at any $\theta \in \mathring \Theta$ is constant and equal to $W_m$.

\begin{figure}[t!]
  \centering
  \includegraphics*[width=0.8\linewidth]{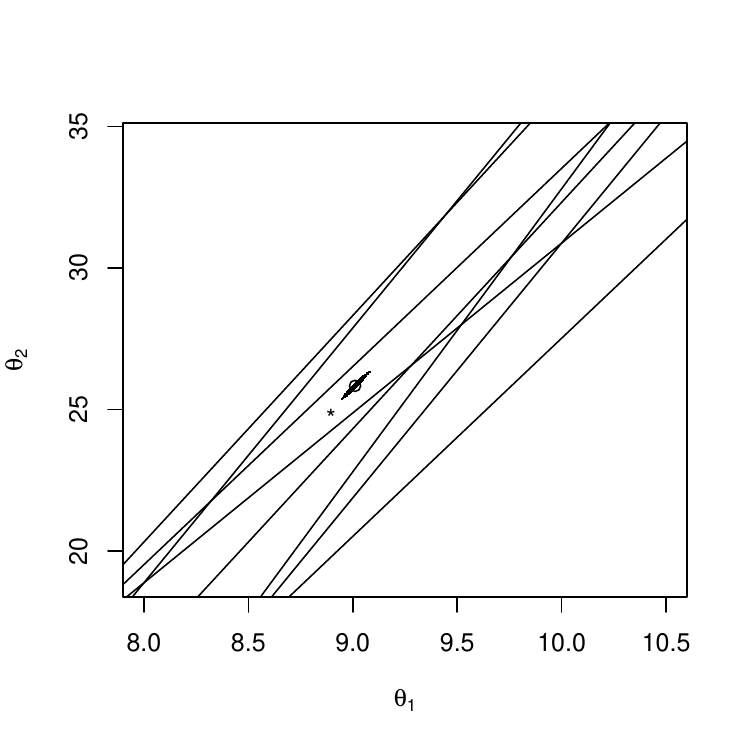}
  \caption{Lines representing the linear inequalities appearing in the definition of $\Theta$ in~\eqref{eq:Theta:moments}. The symbol `$\star$' represents the point $(8.896, 24.8704) \in \mathring \Theta$ corresponding to the probability vector of the binomial distribution with parameters $m-1$ and 0.4. For the squared Hellinger divergence, the vector $\theta_0$ at which $D_\phi(\Sc(\cdot) \mid q_0)$ reaches its minimum is represented by the symbol `o'. The small oblique cloud of points around `o' consists of realizations of the value $\theta_n$ at which the function $D_\phi(\Sc(\cdot) \mid q_n)$ reaches its minimum.}
  \label{fig:feasible}
\end{figure}

As an illustration, we take $m=5$, $\Xc = \{0,\dots,m-1\}$, $r=2$ and $\mu = (1, 1.6, 3.52)$. The linear inequalities appearing in the definition of $\Theta$ in~\eqref{eq:Theta:moments} are then represented in Figure~\ref{fig:feasible}. Note that $\Mc = \Sc(\bar \Theta)$ is nonempty as $\Sc(\mathring \Theta)$ contains the probability vector of the binomial distribution with parameters $m-1$ and 0.4 (whose raw moments of order 1, 2, 3 and 4 are 1.6, 3.52,  8.896 and 24.8704, respectively). This also implies that $(8.896, 24.8704) \in \mathring \Theta$. The latter point is represented by the symbol `$\star$'  in Figure~\ref{fig:feasible}. For $q_0$, we take $(0.35, 0.3, 0.15, 0.1, 0.1)$. Note that $q_0$ does not belong to $\Mc$ as the corresponding raw moments of order 1, 2, 3 and 4 are 1.3, 3.4, 10.6 and 36.4, respectively. As in the previous application, as \phidivergence, we consider Pearson's $\chi^2$ divergence, the squared Hellinger divergence and the Kullback--Leibler divergence. In all three cases, we use the \texttt{constrOptim()} function of the \textsf{R} statistical environment to find the vector $\theta_0$ at which $D_\phi(\Sc(\cdot) \mid q_0)$ reaches its minimum. For Pearson's $\chi^2$ divergence (resp.\ the squared Hellinger divergence, the Kullback--Leibler divergence), after rounding to two decimal places, we obtain $(9.18, 26.84)$ (resp.\ $(9.01, 25.84)$, $(9.07, 26.19)$). For the squared Hellinger divergence, $\theta_0$ is represented by the symbol `o' in Figure~\ref{fig:feasible}. Additional calculations show that, in all three cases, $\Sc(\theta_0) \in (0,1)^m$, which implies that Condition~\ref{cond:support} holds. Note that, from Table~\ref{tab:phi}, Condition~\ref{cond:phi:strong:convex} holds for all three divergences under consideration and that, from~\eqref{eq:Theta:moments} and~\eqref{eq:affine:moments}, Condition~\ref{cond:S:affine} holds as well. Proposition~\ref{prop:cond:support} and Remark~\ref{rem:support} then imply that Condition~\ref{cond:support} automatically holds for the Kullback--Leibler and the squared Hellinger divergences. For Pearson's $\chi^2$ divergence, however, the additional calculations mentioned above cannot be dispensed with.

For the $\chi^2$ divergence (resp.\ the Kullback--Leibler divergence and the squared Hell\-inger divergence), we can apply Corollary~\ref{cor:diff:convex} (resp.\ Corollary~\ref{cor:diff:linear:ineq}) and compute $J[\Sc^*](q_0)$ as given therein as well as $\Sigma = J[\Sc^*](q_0) \Sigma_{q_0} J[\Sc^*](q_0)^\top$. Again, for all three \phidivergences, the asymptotic covariance matrix $\Sigma$ of $\sqrt{n} (\Sc^*(q_n) - \Sc^*(q_0))$ and its approximation $\Sigma_{n,N}$ showed very good agreement. We print their rounded versions below for the squared Hellinger divergence:

\begin{align*}
\input{moments-Hellinger.tex}.
\end{align*}

\change{\subsubsection{$\phi$-projection on a Fréchet class of bivariate probability arrays}
\label{sec:Frechet}}

Let $r,s \in \N^+$. In this last \change{example}, the initial form of the probability vector to be $\phi$-projected is a $r \times s$ matrix $\un{q}_0$. We further assume that all its elements are strictly positive, that is, $\un{q}_0 \in (0,1)^{r \times s}$. Let $a \in (0,1)^r$ and $b \in (0,1)^s$ be two fixed probability vectors. This last illustration is about $\phi$-projecting the bivariate probability array $\un{q}_0$ on the Fréchet class $\un{\Mc} \subset [0,1]^{r \times s}$ of bivariate probability arrays with univariate margins $a$ and $b$, a problem which arises in many areas of probability and statistics \citep[see, e.g.,][]{Csi75,LitWu91,VajVan05,Gee20}. To be able to exploit the results of Section~\ref{sec:convex}, we will need to consider a vectorized version of this problem. 

Let us first define the set $\un{\Mc}$ in the spirit of Condition~\ref{cond:M}. Let
\begin{align*}
  \un{\Theta} = \Bigg\{ \un{\theta} \in [0,1]^{(r-1) \times (s-1)} : &\sum_{i=1}^{r-1} \un{\theta}_{ij} \leq b_j, j \in \{1,\dots,s-1\},  \\
                                                                    &\sum_{j=1}^{s-1} \un{\theta}_{ij} \leq a_j, i \in \{1,\dots,r-1\} \Bigg\}.
\end{align*}
Note that $\un{\Theta}$ is closed and convex. Some thought reveals that the function $\un{\Sc}$ from $\un{\Theta}$ to $[0,1]^{r \times s}$ defined by
\begin{equation*}
  \un{\Sc}(\un{\theta}) = \begin{bmatrix}
    \theta_{11}  & \dots & \theta_{1,s-1} & a_1 - \sum_{j=1}^{s-1} \theta_{1j} \\
    \vdots &  & \vdots & \vdots \\
    \theta_{r-1,1}  & \dots & \theta_{r-1,s-1} & a_{r-1} - \sum_{j=1}^{s-1} \theta_{r-1,j}  \\
    b_1 - \sum_{i=1}^{r-1} \theta_{i1} & \dots & b_{s-1} - \sum_{i=1}^{r-1} \theta_{i,s-1} & a_r + b_s - 1 + \sum_{i=1}^{r-1} \sum_{j=1}^{s-1} \theta_{ij}
  \end{bmatrix}
\end{equation*}
then generates the Fréchet class $\un{\Mc} \subset [0,1]^{r \times s}$ of bivariate probability arrays with univariate margins $a$ and $b$, that is, $\un{\Mc} = \un{\Sc}(\un{\bar \Theta})$. Furthermore, some thought reveals that $\un{\Sc}(\un{\mathring \Theta}) =  \un{\Mc} \cap (0,1)^{r \times s}$ and $\un{\Mc}$ is not empty as $\un{\Sc}(\un{\mathring \Theta})$ contains the bivariate probability array $a b^\top$. Finally, $\un{\Mc}$ is compact (since $\un{S}$ is continuous and $\un{\bar \Theta}$ is compact) and convex.

For $c,d \in \N^+$, let $\vect_{c,d}$ be the vectorization operator, which, given a matrix in $\R^{c \times d}$, returns its column-major vectorization in $\R^{cd}$ and let $\vect_{c,d}^{-1}$ denote the inverse operator. Next, let $\Theta = \vect_{r-1,s-1}(\un{\Theta})$ and let $\Sc$ be the function from $\Theta$ to $[0,1]^{rs}$ defined by
$$
\Sc(\theta) =  \vect_{r,s} \circ \un{\Sc} \circ \vect_{r-1,s-1}^{-1}(\theta), \qquad \theta \in \Theta.
$$
Some thought then reveals that
$$
\Sc(\theta) = A \theta + \Sc(0_{\R^{(r-1)(s-1)}}), \qquad \theta \in \Theta,
$$
where $A$ is the $rs \times (r-1)(s-1)$ matrix given by
\begin{equation*}
  \begin{bmatrix}
    Q  & 0 & \dots & 0 \\
    0 & Q & \dots & 0 \\
    \vdots & \vdots & \ddots & \vdots \\
  0 & 0 & \dots & Q \\
    -Q & -Q & \dots & -Q \\
  \end{bmatrix}
  \text{ with }
  Q =
  \begin{bmatrix}
    1  & 0 & \dots & 0 \\
    0 & 1 & \dots & 0 \\
    \vdots & \vdots & \ddots & \vdots \\
    0 & 0 & \dots & 1 \\
    -1 & -1 & \dots & -1 \\
  \end{bmatrix}
  \in \R^{r \times (r-1)}.
\end{equation*}
Hence, Condition~\ref{cond:M} holds with $m = rs$ and $k = (r-1)(s-1)$, and the resulting set $\Mc = \Sc(\bar \Theta)$ is simply $\vect_{r,s}(\un{\Mc})$. Since $\un{\Sc}(\un{\mathring \Theta}) =  \un{\Mc} \cap (0,1)^{r \times s}$, Condition~\ref{cond:interior} is also satisfied. Finally, Condition~\ref{cond:S:affine} holds as well since $\Theta = \bar \Theta$ is convex and $\Sc$ is affine. We shall thus be able to apply Corollary~\ref{cor:diff:convex} provided Conditions~\ref{cond:phi:strong:convex} and~\ref{cond:support} hold. Alternatively, since $\Theta$ is defined by linear inequalities, we could also rely on Corollary~\ref{cor:diff:linear:ineq} provided Condition~\ref{cond:phi:strong:convex} and $\lim_{x \to 0^+} \phi'(x) = -\infty$ hold. 

Let us now $\phi$-project $\un{q}_0$ on $\un{\Mc}$ or, equivalently, $q_0 = \vect_{r,s}(\un{q}_0)$ on $\Mc$. As an example, we take $r = s = 3$, $a = (0.2, 0.3, 0.5)$, $b = (0.5, 0.25, 0.25)$ and
$$
\un{q}_0 = \begin{bmatrix}
  0.04 & 0.11 & 0.13 \\
  0.10 & 0.07 & 0.08 \\
  0.14 & 0.12 & 0.21 \\
\end{bmatrix}.
$$
To simplify computations, we shall only use the Kullback--Leibler divergence here, as the $I$-projection of a bivariate probability array on a Fréchet class can be conveniently carried out in practice using the iterative proportional fitting procedure (IPFP), also known as Sinkhorn's algorithm or matrix scaling in the literature \cite[see, e.g.,][]{Puk14,Ide16,BroLeu18}. Specifically, we use the function \texttt{Ipfp()} of the \textsf{R} package \texttt{mipfp} \citep{mipfp}. The $I$-projection of $\un{q}_0$ on $\un{\Mc}$, rounded to two decimal places, is found to be
$$
\un{p}_0 = \begin{bmatrix}
  0.06 & 0.08 & 0.06 \\
  0.18 & 0.07 & 0.06 \\
  0.27 & 0.10 & 0.13 \\
\end{bmatrix},
$$
so that the vector $\theta_0$ (rounded to two decimal places) at which $D_\phi(\Sc(\cdot) \mid q_0)$ reaches its minimum on $\mathring \Theta$ is $(0.06, 0.18, 0.08, 0.07)$. Since Condition~\ref{cond:phi:strong:convex} and $\lim_{x \to 0^+} \phi'(x) = -\infty$ hold for the Kullback--Leibler divergence (see Table~\ref{tab:phi}), we can directly apply Corollary~\ref{cor:diff:linear:ineq} and compute $J[\Sc^*](q_0)$ as well as the asymptotic covariance matrix $\Sigma = J[\Sc^*](q_0) \Sigma_{q_0} J[\Sc^*](q_0)^\top$. The rounded version of the latter as well as of its approximation $\Sigma_{n,N}$ are:
{\small \begin{align*}
\input{frechet-Sigma.tex},
\end{align*}
\begin{align*}
\input{frechet-Sigma-n-N.tex},
\end{align*}}
showing again very good agreement. 

Finally, let us briefly explain the slight difference that exists between the above application of Corollary~\ref{cor:diff:convex} to $I$-projections on Fréchet classes and the application of Proposition 2.5 of \cite{KojMar24} to a similar setting resulting in Proposition~4.3 in the same reference. Essentially, instead of obtaining differentiability results for the function $\Sc^*$ in~\eqref{eq:S:*:map} at $q_0$, \cite{KojMar24} provided differentiability results for the function
$$
(q_1,\dots,q_{rs-1}) \mapsto \Sc^*\left(q_1,\dots,q_{rs-1},1-\sum_{i=1}^{rs-1} q_i\right)
$$
at $q_0$. By the chain rule, the Jacobian of the latter function at $q_0$ is simply $J[\Sc^*](q_0) R$, where $J[\Sc^*](q_0)$ is given in Corollary~\ref{cor:diff:convex} and $R$ is the $rs \times (rs-1)$ matrix obtained by adding a row of -1's below the identity matrix $I_{rs-1}$.

\appendix

%%%%%%%%%%%%%%%%%%%%%%%%%%%%%%%%%%%%%%%%%%%%%%%%%%%%%%%%%%%%%%%%%%%%%%%%%%%%%%%%%%%%%%%%%%%%%%%%%%%%%%%%%%% 

\bigskip
\section{Proofs of the results of Section~\ref{sec:prelim}}
\label{proofs:prelim}

\begin{proof}[\bf Proof of Proposition~\ref{prop:D:phi:strong:convex}]
  First, notice that the set $[0,1]^m_{\supp(t)}$ is indeed convex since any convex combination of elements of $[0,1]^m_{\supp(t)}$ is in $[0,1]^m$ and has support included in $\supp(t)$. Next, recall the definition of $f$ in~\eqref{eq:f} and let us first show that Condition~\ref{cond:phi:strong:convex} implies that, for any $w \in (0,\infty)$, $f(\cdot, w)$ is strongly convex on $[0,1]$. Fix $w \in (0,\infty)$ and let $\kappa_\phi(w)$ be the strong convexity constant of the restriction of $\phi$ to $[0,1/w]$. Then, for any $v,v' \in [0,1]$ and $\alpha \in [0,1]$,
\begin{align*}
  f(\alpha v + (1-\alpha) v', w) &= w \phi \left(\alpha \frac{v}{w} + (1-\alpha) \frac{v'}{w} \right) \\
                    &\leq w \alpha \phi \left(\frac{v}{w} \right) + w (1-\alpha) \phi \left( \frac{v'}{w} \right) - w \frac{\kappa_\phi(w)}{2} \alpha (1-\alpha) \left| \frac{v}{w}  - \frac{v'}{w} \right|^2 \\
                    &=  \alpha f(v,w) + (1-\alpha) f(v',w) - \frac{\kappa_\phi(w)}{2 w} \alpha (1-\alpha) | v - v' |^2.
\end{align*}
Hence, for any $w \in (0,\infty)$, $f(\cdot, w)$ is strongly convex on $[0,1]$ with strong convexity constant $\kappa_\phi(w)/w$.

Now, let $t \in [0,1]^m$, $t \neq 0_{\R^m}$, and notice that, for any $s \in [0,1]^m_{\supp(t)}$, all the summands in the expression of $D_\phi(s \mid t)$ in~\eqref{eq:phi:div} corresponding to the $t_i$'s that are not strictly positive are zero because of the definition of $f$ in~\eqref{eq:f}. As a consequence, without loss of generality, we can assume that $t \in (0,1]^m$. Hence, for any  $t \in (0,1]^m$, $s,s' \in [0,1]^m$ and $\alpha \in [0,1]$,
\begin{align*}
  D_\phi(\alpha s &+ (1-\alpha) s' \mid t) = \sum_{i=1}^m f(\alpha s_i + (1-\alpha) s'_i, t_i) \\
                                    &\leq \alpha \sum_{i=1}^m f(s_i, t_i) + (1-\alpha)   \sum_{i=1}^mf(s'_i, t_i) - \alpha (1-\alpha)  \sum_{i=1}^m \frac{\kappa_\phi(t_i)}{2 t_i}  | s_i - s'_i |^2 \\
                                    &\leq \alpha D_\phi(s \mid t) + (1-\alpha)   D_\phi(s' \mid t) -  \alpha (1-\alpha)  \| s - s' \|_2^2 \min_{i \in \{1,\dots,m\}} \frac{\kappa_\phi(t_i)}{2 t_i}.
\end{align*}
Hence, for any $t \in (0,1]^m$, the function $D_\phi(\cdot \mid t)$ is strongly convex on $[0,1]^m$ with strong convexity constant $\min_{i \in \{1,\dots,m\}} \kappa_\phi(t_i) / t_i$.
\end{proof}

\begin{proof}[\bf Proof of Lemma~\ref{lem:phi:strict:convex}]
  Let $x,y$ in $[0, \infty)$ such that $x \neq y$. Notice that any $w \leq 1 / \max \{x, y\}$ is such that $x$ and $y$ are in $[0, 1/w]$. Fix one such $w$ and let $\kappa_\phi(w)$ be the strong convexity constant of the restriction of $\phi$ to $[0,1/w]$. Then, for any $\alpha \in (0,1)$,
  $$
  \phi(\alpha x + (1-\alpha)y) \leq \alpha\phi(x) + (1-\alpha)\phi(y) - \frac{1}{2} \kappa_\phi(w) \alpha(1-\alpha) | x - y |^2 < \alpha\phi(x) + (1-\alpha)\phi(y),
  $$
  since $\frac{1}{2} \kappa_\phi(w) \alpha(1-\alpha) | x - y |^2 > 0$. As $x,y$ are arbitrary in $[0, \infty)$, this establishes the strict convexity of $\phi$. \qed
\end{proof}

%%%%%%%%%%%%%%%%%%%%%%%%%%%%%%%%%%%%%%%%%%%%%%%%%%%%%%%%%%%%%%%%%%%%%%%%%%%%%%%%%%%%%%%%%%%%%%%%%%%%%%%%%%%

\section{Proofs of the results of Section~\ref{sec:param}}
\label{proofs:param}

\begin{proof}[\bf Proof of Proposition~\ref{prop:cont:vartheta:t0}]
The proof is an adaption of the proof of Theorem~3.3~(iii) in \cite{GieRef17}. Let $(t_n)$ be a sequence in $[0,1]^m$ such that $\lim_{n \to \infty} t_n = t_0$. For $n$ large enough, $\theta_n = \vartheta^*(t_n)$ is well-defined via~\eqref{eq:theta:map} and $s_n = \Sc(\theta_n)$ is the $\phi$-projection of $t_n$ on $\Mc = \Sc(\bar \Theta)$. Since $\bar \Theta$ is a compact subset of $\R^k$, by the Bolzano--Weierstrass theorem, the sequence $(\theta_n)$ in $\bar \Theta$ has a convergent subsequence which converges to an element of $\bar \Theta$. Let $(\theta_{\eta_n})$ be such a subsequence and let $\lim_{n \to \infty} \theta_{\eta_n} = \theta^{**}$. By continuity of $\Sc$, we immediately obtain that $\lim_{n \to \infty} \Sc(\theta_{\eta_n}) = \Sc(\theta^{**})$. Then, using the fact that every subsequence of $(t_n)$ converges to $t_0$ and the lower semicontinuity of $D_\phi$ (see Proposition~\ref{prop:properties:D:phi}~(i)), we obtain that, for any $\theta \in \bar \Theta$,
  \begin{align*}
    D_\phi(\Sc(\theta^{**}) \mid t_0) &= D_\phi(\lim_{n \to \infty} \Sc(\theta_{\eta_n}) \mid \lim_{n \to \infty} t_n) = D_\phi(\lim_{n \to \infty} \Sc(\theta_{\eta_n}) \mid \lim_{n \to \infty} t_{\eta_n})\\
    &\leq \liminf_{n \to \infty} D_\phi(\Sc(\theta_{\eta_n}) \mid t_{\eta_n}) \leq \liminf_{n \to \infty} D_\phi(\Sc(\theta) \mid t_{\eta_n}),
  \end{align*}
  where we have used the fact that, for $n$ large enough, $D_\phi(\Sc(\theta_{\eta_n}) \mid t_{\eta_n}) \leq D_\phi(\Sc(\theta) \mid t_{\eta_n})$ since $\theta_{\eta_n} = \vartheta^*(t_{\eta_n})$. Finally, using the continuity of $D_\phi$ with respect to its second argument (see Proposition~\ref{prop:properties:D:phi}~(ii)), we obtain
  $$
  \liminf_{n \to \infty} D_\phi(\Sc(\theta) \mid t_{\eta_n}) = \lim_{n \to \infty} D_\phi(\Sc(\theta) \mid t_{\eta_n}) = D_\phi(\Sc(\theta) \mid t_0).
  $$
  In other words, we have shown that, for any $\theta \in \bar \Theta$, $D_\phi(\Sc(\theta^{**}) \mid t_0) \leq D_\phi(\Sc(\theta) \mid t_0)$, which, since $\Sc(\vartheta^*(t_0))$ is the unique $\phi$-projection of $t_0$ on $\Mc = \Sc(\bar \Theta)$ by Condition~\ref{cond:unicity:t0}, implies that $\Sc(\theta^{**}) = \Sc(\vartheta^*(t_0))$, and, by the injectivity of $\Sc$, that $\theta^{**} = \vartheta^*(t_0)$. Hence, $\lim_{n \to \infty} \theta_{\eta_n} = \lim_{n \to \infty} \vartheta^*(t_{\eta_n}) = \vartheta^*(t_0)$. Reasoning by contraposition, this implies that $\vartheta^*$ is continuous at $t_0$. \qed
\end{proof}

\begin{proof}[\bf Proof of Lemma~\ref{lem:J2}]
Let us first state a formula that we will need later in the proof. Having in mind the definitions given above Lemma~\ref{lem:J2}, let $g : U \rightarrow \R^{b \times c} $ and $h : U \rightarrow \R^{c \times d}$, where $U$ is an open subset of $\R^a$, be continuously differentiable on $U$. Then, the product function $gh$ given by $(gh)(x) = g(x) h(x)$, $x \in U$, is well-defined and, using for instance Theorem T4.3 in Table IV of \cite{Bre78}, we have
\begin{equation}
  \label{eq:J:prod}
J[gh](x) = \left( h(x)^\top \otimes I_b \right) J[g](x) + \left( I_d \otimes g(x) \right) J[h](x), \qquad x \in U.
\end{equation}

We shall now proceed with the proof. From~\eqref{eq:f} and~\eqref{eq:phi:div}, for any $t \in (0,1)^m$ and any $s \in [0,1]^m$,
$$
D_\phi(s \mid t) = \sum_{i=1}^m t_i \phi \left( \frac{s_i}{t_i} \right),
$$
implying that, under Condition~\ref{cond:phi:diff}, for any $t \in (0,1)^m$, $D_\phi(\cdot \mid t)$ is twice continuously differentiable on $(0,1)^m$. Then, since $\Sc$ is assumed to be twice continuously differentiable on $\mathring \Theta$ and $\Sc(\mathring \Theta) \subset (0,1)^m$ by Condition~\ref{cond:M}, for any $t \in (0,1)^m$, $D_\phi(\Sc(\cdot) \mid t)$ is twice continuously differentiable on $\mathring \Theta$. Next, fix $t \in (0,1)^m$ and consider the function from  $(0,1)^m$ to $\R^m$ defined by
\begin{equation*}
  %\label{eq:Phi':t}
  \Phi_t'(s) = J[D_\phi(\cdot \mid t)](s) = \left(\phi' \left( \frac{s_1}{t_1} \right), \dots, \phi' \left( \frac{s_m}{t_m} \right) \right), \qquad s \in (0,1)^m.
\end{equation*}
Then, by the chain rule, for any $\theta \in \mathring \Theta$, 
\begin{align*}
  J[D_\phi(\Sc(\cdot) \mid t)](\theta) &= J[D_\phi(\cdot \mid t)](\Sc(\theta)) J[\Sc](\theta) =  \Phi'_t(\Sc(\theta))^\top  J[\Sc](\theta).
\end{align*} 
Note in passing that
$$
J[\Phi_t'](s) = J_2[D_\phi(\cdot \mid t)](s) = \diag \left(\frac{1}{t_1}\phi'' \left( \frac{s_1}{t_1} \right), \dots, \frac{1}{t_m}\phi'' \left( \frac{s_m}{t_m} \right) \right), \qquad s \in (0,1)^m.
$$
Finally, from~\eqref{eq:J:prod} and the chain rule, for any $\theta \in \mathring \Theta$,
\begin{align*}
  J_2[D_\phi(\Sc(\cdot) \mid t)](\theta) =& J \big[ J[D_\phi(\Sc(\cdot) \mid t)] \big](\theta) = J \big[ \Phi'_t(\Sc(\cdot))^\top  J[\Sc] \big](\theta) \\
  =&  \left( J[\Sc](\theta)^\top \otimes I_1 \right) J \left[ \Phi'_t(\Sc(\cdot)) \right](\theta) + \left( I_k \otimes \Phi'_t(\Sc(\theta))^\top \right) J_2[\Sc] (\theta) \\
  =&  J[\Sc](\theta)^\top J \left[ \Phi'_t \right](\Sc(\theta)) J[\Sc](\theta) + \left( I_k \otimes \Phi'_t(\Sc(\theta))^\top \right) J_2[\Sc] (\theta) \\
                                            =&  J[\Sc](\theta)^\top  \diag \left(\frac{1}{t_1}\phi'' \left( \frac{\Sc_1(\theta)}{t_1} \right), \dots, \frac{1}{t_m}\phi'' \left( \frac{\Sc_m(\theta)}{t_m} \right)  \right) J[\Sc](\theta) \\
                                          &+ \left( I_k \otimes \left(\phi' \left( \frac{\Sc_1(\theta)}{t_1} \right), \dots, \phi' \left( \frac{\Sc_m(\theta)}{t_m} \right) \right)^\top \right) J_2[\Sc] (\theta).
\end{align*}
\qed
\end{proof}

\begin{proof}[\bf Proof of Theorem~\ref{thm:diff:vartheta}]
As verified in the proof of Lemma~\ref{lem:J2}, under the conditions considered from Section~\ref{sec:diff} onwards, we have that, for any $t \in (0,1)^m$, the function $D_\phi(\Sc(\cdot) \mid t)$ is twice continuously differentiable on $\mathring \Theta$.  Next, let $F$ be the function from $(0,1)^m \times \mathring \Theta$ to $\R^k$ defined by $F(t, \theta) = J[D_\phi(\Sc(\cdot) \mid  t )](\theta)$, $t \in (0,1)^m$, $\theta \in \mathring \Theta$. From the chain rule, we have that, for any $t \in (0,1)^m$ and $\theta \in \mathring \Theta$,
\begin{equation}
  \label{eq:F}
  F(t, \theta) = J[D_\phi(\cdot \mid t)](\Sc(\theta)) J[\Sc](\theta) =  \left(\phi' \left( \frac{\Sc_1(\theta)}{t_1} \right), \dots, \phi' \left( \frac{\Sc_m(\theta)}{t_m} \right) \right)^\top  J[\Sc](\theta),
\end{equation}
implying that $F$ is continuously differentiable on the open subset $(0,1)^m \times \mathring \Theta$ of $\R^{m+k}$.

From Condition~\ref{cond:unicity:interior:t0}, for any $t \in \Nc(t_0)$, the function $D_\phi(\Sc(\cdot) \mid t)$ from $\mathring \Theta$ to $\R$ reaches its unique minimum at $\vartheta^*(t) \in \mathring \Theta$, where the function $\vartheta^*$ is defined in~\eqref{eq:theta:map:2}. First order necessary optimality conditions then imply that
\begin{equation}
  \label{eq:first:order}
  J[D_\phi(\Sc(\cdot) \mid  t )](\vartheta^*(t)) = F(t, \vartheta^*(t)) = 0_{\R^k} \quad \text{for all } t \in \Nc(t_0).
\end{equation}

Since $F(t_0, \vartheta^*(t_0)) = 0_{\R^k}$ by~\eqref{eq:first:order} and $J_2[D_\phi(\Sc(\cdot) \mid  t_0 )](\vartheta^*(t_0)) = J[F(t_0, \cdot)](\vartheta^*(t_0))$ is invertible by Condition~\ref{cond:invertibility}, we can apply the implicit function theorem \citep[see, e.g.,][Theorem 17.6, p 450]{Fit09} to obtain that there exists $r \in (0,\infty)$, an open ball $\Bc \subset (0,1)^m$ of radius $r$ centered at $t_0$ and a continuously differentiable function $g:\Bc \to \R^k$ such that
$$
F(t, g(t)) = 0_{\R^k} \qquad \text{for all } t \in \Bc,
$$
and, whenever $\|t - t_0 \| < r$, $\| \theta - \vartheta^*(t_0) \| < r$ and $F(t, \theta) = 0$, then $\theta = g(t)$. Moreover,
$$
J[F(\cdot,g(t)](t)  + J[F(t,\cdot)](g(t)) J[g](t) = 0_{\R^k} \qquad \text{for all } t \in \Bc.
$$
Now, under Condition~\ref{cond:unicity:interior:t0}, Proposition~\ref{prop:cont:vartheta:t0} states that the function $\vartheta^*$ in~\eqref{eq:theta:map:2} is continuous at~$t_0$. Hence, there exists an open neighborhood $\Bc' \subset \Bc$ of $t_0$ such that, for any $t \in \Bc'$,  $\| \vartheta^*(t) - \vartheta^*(t_0) \|  < r$. Besides, by~\eqref{eq:first:order}, $F(t, \vartheta^*(t)) = 0_{\R^k}$ for all $t \in \Bc' \cap \Nc(t_0)$. Therefore, by the implicit function theorem, for any $t \in \Bc' \cap \Nc(t_0)$, $\vartheta^*(t) = g(t)$ and thus $\vartheta^*$ is continuously differentiable at $t_0$ with Jacobian matrix at $t_0$ given by
$$
J[\vartheta^*](t_0) = - J[F(t_0,\cdot)](\vartheta^*(t_0))^{-1} J[F(\cdot,\vartheta^*(t_0)](t_0).
$$
Consider the continuously differentiable function from $(0,1)^m$ to $\R^m$ defined by
$$
\Psi'(t) = \left(\phi' \left( \frac{\Sc_1(\vartheta^*(t_0))}{t_1} \right), \dots, \phi' \left( \frac{\Sc_m(\vartheta^*(t_0))}{t_m} \right) \right), \qquad t \in (0,1)^m,
$$
and note that $J[\Psi'](t_0)$ is equal to $-\Delta(t_0)$, where $\Delta(t_0)$ is given in~\eqref{eq:Delta:t0}. The expression in~\eqref{eq:J:vartheta:map} finally follows from the fact that $J[F(t_0, \cdot)](\vartheta^*(t_0)) = J_2[D_\phi(\Sc(\cdot) \mid  t_0 )](\vartheta^*(t_0))$ and the fact that, by~\eqref{eq:J:prod} and~\eqref{eq:F},
\begin{align*}
  J[F(\cdot,\vartheta^*(t_0)](t_0) &= J \big[ \Psi'(\cdot)^\top  J[\Sc](\vartheta^*(t_0)) \big](t_0) = \big( J[\Sc](\vartheta^*(t_0))^\top \otimes I_1 \big) J[\Psi'](t_0).
\end{align*}
\qed
\end{proof}

\begin{proof}[\bf Proof of Lemma~\ref{lem:cond:unicity:interior}]
  From Conditions~\ref{cond:unicity:t0} and~\ref{cond:support}, there exists an open neighborhood $\Nc(t_0) \subset (0,1)^m$ of $t_0$ such that, for any $t \in \Nc(t_0)$, the $\phi$-projection $s^*$ of $t$ on $\Mc$ exists, is unique and belongs to $(0,1)^m$. Furthermore, Condition~\ref{cond:M} implies that the function $\Sc$ is a bijection between $\bar \Theta$ and $\Mc$. Under Condition~\ref{cond:interior}, we immediately obtain that it is also a bijection between $\mathring \Theta$ and $\Mc \cap (0,1)^m$. Hence, for any $t \in \Nc(t_0)$, the unique $\phi$-projection $s^*= \Sc^*(t) \in \Mc \cap (0,1)^m$ satisfies $s^* = \Sc(\theta^*)$ for some (unique) $\theta^* \in \mathring \Theta$. \qed
\end{proof}

%%%%%%%%%%%%%%%%%%%%%%%%%%%%%%%%%%%%%%%%%%%%%%%%%%%%%%%%%%%%%%%%%%%%%%%%%%%%%%%%%%%%%%%%%%%%%%%%%%%%%%%%%%%

\section{Proofs of the results of Section~\ref{sec:convex}}
\label{proofs:convex}

\begin{proof}[\bf Proof of Proposition~\ref{prop:cond:unicity:t0}] 
  First recall that, by continuity of $\Sc$ and compactness of $\bar \Theta$, $\Mc = \Sc(\bar \Theta)$ is a compact subset of $[0,1]^m$. Note also that, for any $t \in (0,1)^m$, $[0,1]^m_{\supp(t)} = [0,1]^m$. Then, by Proposition~\ref{prop:existence:phi:proj} and, since Condition~\ref{cond:phi:strict:convex} holds, by Proposition~\ref{prop:phi:proj:convex}~(i), for any $t \in (0,1)^m$, the $\phi$-projection $s^*$ of $t$ on $\Mc = \Sc(\bar \Theta)$ exists and is unique (which implies that Condition~\ref{cond:unicity:t0} holds). \qed
\end{proof}

\begin{proof}[\bf Proof of Proposition~\ref{prop:S:strong:invert}]
  Fix $t \in (0,1)^m$. As already verified in the proof of Lemma~\ref{lem:J2}, the function $D_\phi(\Sc(\cdot) \mid t)$ is twice continuously differentiable on $\mathring \Theta$. From Theorem~2.1.11 in \cite{Nes04}, we have that the strong convexity of the function $D_\phi(\Sc(\cdot) \mid t)$ from $\bar \Theta$ to $\R$  is equivalent to the fact that, for any $\theta \in \mathring \Theta$, $J_2[D_\phi(\Sc(\cdot) \mid t)](\theta) - \kappa(t) I_k$ is positive semi-definite, where $\kappa(t)$ is the strong convexity constant of the function $D_\phi(\Sc(\cdot) \mid t)$. This then immediately implies that, for any $t \in (0,1)^m$ and $\theta \in \mathring \Theta$, $J_2[D_\phi(\Sc(\cdot) \mid t)](\theta)$ is positive definite. Since, under Condition~\ref{cond:unicity:interior:t0}, $\vartheta^*$ can be expressed as in~\eqref{eq:theta:map:2}, we immediately obtain $J_2[D_\phi(\Sc(\cdot) \mid t_0)](\vartheta^*(t_0))$ is positive definite, that is, that Condition~\ref{cond:invertibility} holds. \qed
\end{proof}

\begin{proof}[\bf Proof of Lemma~\ref{lem:cond:S}]
From Condition~\ref{cond:S:affine}, for any $\theta,\theta' \in \bar \Theta$ and $\alpha \in [0,1]$,
\begin{align*}
  \Sc(\alpha \theta + (1-\alpha) \theta') &=  A (\alpha \theta + (1-\alpha) \theta') + \gamma = \alpha \Sc(\theta) + (1-\alpha) \Sc(\theta').
\end{align*}
Next, since Condition~\ref{cond:phi:strong:convex} holds, from Proposition~\ref{prop:D:phi:strong:convex}, for any $\theta,\theta' \in \bar \Theta$, $t \in (0,1)^m$ and $\alpha \in [0,1]$,
\begin{align*}
  D_\phi(\Sc(\alpha \theta &+ (1-\alpha) \theta') \mid t) = D_\phi(\alpha \Sc(\theta) + (1-\alpha) \Sc(\theta') \mid t) \\
                    &\leq \alpha D_\phi(\Sc(\theta) \mid t) + (1-\alpha) D_\phi(\Sc(\theta') \mid t ) - \frac{\kappa_{D_\phi(\cdot \mid t)}}{2} \alpha (1-\alpha)  \| \Sc(\theta) - \Sc(\theta') \|_2^2 \\
                    &= \alpha D_\phi(\Sc(\theta) \mid t) + (1-\alpha) D_\phi(\Sc(\theta') \mid t) - \frac{\kappa_{D_\phi(\cdot \mid t)}}{2} \alpha (1-\alpha)  \| A(\theta - \theta') \|_2^2 \\
                    &\leq \alpha D_\phi(\Sc(\theta) \mid t) + (1-\alpha) D_\phi(\Sc(\theta') \mid t) - \frac{\kappa_{D_\phi(\cdot \mid t)}}{2}  \alpha (1-\alpha)  c \| \theta - \theta' \|_2^2,
\end{align*}
where $\kappa_{D_\phi(\cdot \mid t)}$ is the strong convexity constant of the function $D_\phi(\cdot \mid t)$ and $c$ is a constant that depends on $A$ such that, for any $x \in \R^k$, $\| Ax \|_2^2 \geq c \| x \|_2^2$. To prove that, for any $t \in (0,1)^m$, the function $D_\phi(\Sc(\cdot) \mid t)$ from $\bar \Theta$ to $\R$ is strongly convex, it remains to show that the constant $c$ can be chosen strictly positive. Let $S^{k-1} = \{ x \in \R^k : \| x \|_2 = 1\}$ be the unit sphere in $\R^k$ and let $c = \min_{x \in S^{k-1}} \| Ax \|_2^2$. Then $c \geq 0$ and, as expected, for any $x \in \R^k$,
$$
\| Ax \|_2^2 =  \| x \|_2^2 \left\| A \frac{x}{\| x \|_2} \right\|_2^2  \geq c  \| x \|_2^2.
$$
To prove that $c > 0$, we proceed by contradiction. Assume that $c = 0$. Then, there would exist $x \in S^{k-1}$ such that $Ax = 0$. Since $\Sc$ is assumed to be affine and injective, $x \mapsto Ax$ is a linear injective function and $Ax = 0$ implies that $x = 0_{\R^k}$. This is a contradiction since $0_{\R^k} \not \in S^{k-1}$. Hence, for any $t \in (0,1)^m$, the function $D_\phi(\Sc(\cdot) \mid t)$ from $\Theta$ to $\R$ is strongly convex. \qed
\end{proof}

\begin{proof}[\bf Proof of Corollary~\ref{cor:diff:convex}]
  To show the result, it suffices to verify that the conditions of Theorem~\ref{thm:diff:vartheta} and Corollary~\ref{cor:diff:S:*}, namely Conditions~\ref{cond:unicity:interior:t0} and~\ref{cond:invertibility}, are satisfied. From Proposition~\ref{prop:cond:unicity:t0}, we know that Condition~\ref{cond:unicity:t0} holds. Lemma~\ref{lem:cond:unicity:interior} then implies that Condition~\ref{cond:unicity:interior:t0} holds since Conditions~\ref{cond:interior} and~\ref{cond:support} are assumed to be satisfied. From Lemma~\ref{lem:cond:S}, we have that Conditions~\ref{cond:phi:strong:convex} and~\ref{cond:S:affine} imply Condition~\ref{cond:S:strong}, which implies Condition~\ref{cond:invertibility} by Proposition~\ref{prop:S:strong:invert}. \qed
\end{proof}

\begin{proof}[\bf Proof of Proposition~\ref{prop:cond:support}]
  Recall that the fact that, for any probability vector $q \in (0,1)^m$, the $\phi$-projection $p^*$ of $q$ on $\Mc = \Sc(\bar \Theta)$ exists and is unique follows from Proposition~\ref{prop:cond:unicity:t0}. Note also that the function $\Sc$ defined in Condition~\ref{cond:M} is a bijection between $\bar \Theta$ and $\Mc = \Sc(\bar \Theta)$. When it is affine (that is, under Condition~\ref{cond:S:affine}), an expression of its inverse is given by Lemma~\ref{lem:S:inv} below. Specifically, we then have that
  \begin{equation*}
    %\label{eq:S:inv}
    \Sc^{-1}(p) = (A^\top A)^{-1}A^\top (p - \gamma), \qquad p \in \Mc.
  \end{equation*}
  This implies that the set $\Mc$ can be equivalently expressed as
  $$
  \Mc = \left\{ p \in [0,1]^m : \sum_{i=1}^m p_i = 1,  (A^\top A)^{-1}A^\top (p - \gamma) \in \bar \Theta \right\},
  $$
  and, since $\Theta$ is defined by linear inequalities, that $\Mc$ is defined by linear inequalities. Some thought then reveals that there exists a $d \in \N^+$, $\alpha_1, \dots, \alpha_d \in \R$ and $d$ functions $g_1, g_2, \dots, g_d$ from $\{1,\dots,m\}$ to $\R$ such that
  $$
  \Mc = \left\{ p \in [0,1]^m : \sum_{i = 1}^m p_i  = 1, \sum_{i = 1}^m p_i g_j(i) \geq \alpha_j, j \in \{1, \dots, d\} \right\}.
  $$
  Setting $f_j(i) = g_j(i) - \alpha_j$ for $i \in \{1, \dots, m\}$ and $j \in \{1, \dots, d\}$, it is easy to verify that $\Mc$ can be equivalently rewritten as 
  $$
  \Mc = \left\{ p \in [0,1]^m : \sum_{i = 1}^m p_i  = 1, \sum_{i = 1}^m p_i f_j(i) \geq 0, j \in \{1, \dots, d\} \right\}.
  $$
  Theorem~2~(c) of \cite{Rus87} and the remark following it then imply that, for any probability vector $q \in (0,1)^m$, there exist $c = (c_0, \dots, c_d) \in \mathbb{R}^{d + 1}$ such that, for any $i \in \{1,\dots,m\}$,
  $$
  p_i^* = q_i \, \phi'^{-1}\left( c_0 + \sum_{j = 1}^d c_j f_j(i) \right),
  $$
  where $p^* = \Sc^*(q)$ is the $\phi$-projection of $q$ on $\Mc$. If $\lim_{x \to 0^+} \phi'(x) = - \infty$, $\phi'^{-1}(x) > 0$ for all $x \in (0,\infty)$ and the previous centered display immediately implies that $\Sc^*(q) \in (0,1)^m$ whenever $q \in (0,1)^m$. \qed
\end{proof}

The following lemma is necessary for proving Proposition~\ref{prop:cond:support}. 

\begin{lem}
  \label{lem:S:inv}
  Under Condition~\ref{cond:S:affine}, the inverse $\Sc^{-1}$ of $\Sc : \bar \Theta \rightarrow \Mc$ can be expressed as
  \begin{equation}
    \label{eq:S:inv:pseudo}
    \Sc^{-1}(s) = (A^\top A)^{-1}A^\top (s - \gamma), \qquad s \in \Mc.
  \end{equation}
\end{lem}

\begin{proof}%[\bf Proof of Lemma~\ref{lem:S:inv}]
  Since $A$ is of full rank because of the injectivity of $\Sc$, we know, for instance from \citet[Theorem 5, p 48]{BenGre03}, that $A^\top A$ is invertible. Moreover, for instance from expression~(6.13) in \cite{TreBau97}, we have that the orthogonal projection of $x \in \R^m$ on $\im(A) = \{ y \in \R^m : y = Az, z \in \R^k \}$ can be computed as $A (A^\top A)^{-1} A^\top x$. This implies that
  \begin{equation}
    \label{eq:S:inv:pseudo:I2}
    A(A^\top A)^{-1}A^\top y = y, \qquad \text{for all } y \in \im(A).
  \end{equation}
  We shall use these facts to prove that~\eqref{eq:S:inv:pseudo} provides an expression of the inverse of the function $\Sc: \bar \Theta \rightarrow \Mc$. %Specifically, we need to verify that, for any $\theta \in \bar\Theta$, $\Sc^{-1}( \Sc (\theta) ) = \theta$ and that, for any $s \in \Mc$, $\Sc( \Sc^{-1} (s) ) = s $.
  Let $\theta \in \bar \Theta$. From the invertibility of $A^\top A$, we have that
  $$
  \Sc^{-1}( \Sc (\theta) ) = \Sc^{-1}(A\theta + \gamma) = (A^\top A)^{-1}A^\top A\theta = \theta.
  $$
  Next, let $s \in \Mc$. Then, using~\eqref{eq:S:inv:pseudo:I2},
  \begin{align*}
    \Sc( \Sc^{-1} (s) ) &= \Sc( (A^\top A)^{-1}A^\top (s - \gamma) ) = A (A^\top A)^{-1}A^\top (s - \gamma) + \gamma  = (s - \gamma) + \gamma = s,
  \end{align*}
  since $s - \gamma$ belongs to $\im(A)$. \qed
\end{proof}

\begin{proof}[\bf Proof of Corollary~\ref{cor:diff:linear:ineq}]
  To show the result, we shall check that the conditions of Corollary~\ref{cor:diff:convex} are satisfied. It thus just remains to verify Condition~\ref{cond:support}. The latter follows from Proposition~\ref{prop:cond:support} since attention is restricted to probability vectors. \qed
\end{proof}

\begin{proof}[\bf Proof of Proposition~\ref{prop:linear:eq}]
  The fact that $\Mc$ is defined by linear constraints is equivalent to the existence of a $d \times m$ matrix $B$ such that $\Mc = \left\{ s \in [0,1]^m : Bs = \alpha \right\}$. Because $\Mc$ is convex and not reduced to a singleton, it contains an infinity of elements. Since $\Mc \cap (0,1)^m$ is nonempty, we can choose $s_0 \in \Mc \cap (0,1)^m$. Then, for any $s \in \Mc$, $s - s_0 \in \ker{(B)}$, where $\ker{(B)} = \{ x \in \R^m : Bx = 0_{\R^d}\}$. Since $\ker{(B)} \neq \{0_{\R^m}\}$, there exists a strictly positive integer $k < m$ and $\delta_1, \dots, \delta_k \in \R^m$ that form a basis of $\ker{(B)}$. Let $A$ be the matrix whose column vectors are $\delta_1, \dots, \delta_k$ and let
  \begin{equation}
    \label{eq:linconstr:Theta}
    \Theta = \left\{ \theta \in \R^k : 0_{\R^m} \leq s_0 + A\theta \leq 1_{\R^m} \right\}.
  \end{equation}
  Since, for any $s \in \Mc$ (resp.\ $\Mc \cap (0,1)^m$), $s - s_0 \in \ker{(B)}$, there exists a unique $\theta \in \R^k$ such that $s - s_0 = A\theta$ and, as $s = s_0 + A\theta$ belongs to $[0,1]^m$ (resp.\ $(0,1)^m$), $\theta$ necessarily belongs to $\bar \Theta$ (resp.\ $\mathring \Theta$). In other words, 
  \begin{equation}
    \label{eq:Mc:repr}
    \forall \, s \in \Mc \text{ (resp.\ $\Mc \cap (0,1)^m$)}, \text{ there exists a unique } \theta \in \bar \Theta \text{ (resp.\ $\mathring \Theta$)} \text{ s.t.\ } s = s_0 + A \theta.
  \end{equation}

  To complete the proof, we shall show that Conditions~\ref{cond:M},~\ref{cond:interior} and~\ref{cond:S:affine} are satisfied for $\Theta$ in~\eqref{eq:linconstr:Theta} and the function $\Sc$ defined by $\Sc(\theta) = A\theta + s_0$, $\theta \in \bar\Theta$. Condition~\ref{cond:S:affine} clearly holds since $\bar \Theta$ is convex and $\Sc$ has the right form with $\gamma = s_0$. Let us next verify that Condition~\ref{cond:M} is satisfied. Note that $\Theta$ is clearly bounded. Furthermore, from~\eqref{eq:linconstr:Theta}, the fact that $s_0 \in \Mc \cap (0,1)^m$ is equivalent to $0_{\R^m} < s_0 + A0_{\R^k}  < 1_{\R^m}$ implies that $0_{\R^k} \in \mathring \Theta$ and thus that $\mathring \Theta \neq \emptyset$. The fact that $\Sc$ is a bijection between $\bar \Theta$ and~$\Mc$ follows, on one hand, from the fact that, for any $\theta \in \bar \Theta$, $\Sc(\theta) \in \Mc$ (since $B(A\theta+s_0) = BA\theta+ Bs_0= 0_{\R^d} + \alpha$) and, on the other hand, from \eqref{eq:Mc:repr} (since, for any $s \in \Mc$, there exists a unique $\theta \in \bar \Theta$ such that $\Sc(\theta) = s$). %To verify the injectivity of $\Sc$, notice that, for any $\theta, \theta' \in \bar \Theta$, $\Sc(\theta) = \Sc(\theta')$ is equivalent $A \theta = A \theta'$, which implies that $\theta = \theta'$ since $x \mapsto A x$ is injective as $A$ is of full rank.
  In a similar way, Condition~\ref{cond:interior} is a consequence of the fact that, for any $\theta \in \mathring \Theta$, $\Sc(\theta) \in \Mc \cap (0,1)^m$ and~\eqref{eq:Mc:repr}.
\end{proof}

\change{
\section{Proofs of the results of Section \ref{sec:app}}
\label{proofs:IF}

\begin{proof}[\bf Proof of Corollary~\ref{cor:IF}]
  Fix $j \in \{1,\dots,m\}$. Computing~\eqref{eq:infl} amounts to computing $J[\vartheta^* \circ q^{[j]}_{0}](0)$, the Jacobian matrix of $\vartheta^* \circ q^{[j]}_{0}$ at 0. Clearly, $q_0^{[j]}$ in~\eqref{eq:tj} is differentiable at $0$ with Jacobian matrix at $0$ given by $(e_j - q_0)$. Under Condition~\ref{cond:unicity:interior:t0} with $t_0 = q_0$, $\vartheta^*$ can be expressed as in~\eqref{eq:theta:map:2} with $t_0 = q_0$. Assuming additionally Condition~\ref{cond:invertibility} with $t_0 = q_0$, from Theorem~\ref{thm:diff:vartheta}, $\vartheta^*$ is differentiable at $q_0^{[j]}(0) = q_0$. Its Jacobian matrix at $q_0$, $J[\vartheta^*](q_0)$, is given in~\eqref{eq:J:vartheta:map} with $t_0 = q_0$. By the chain rule \cite[see, for istance,][Theorem 15.39]{Fit09}, $\vartheta^* \circ q^{[j]}_{0}$ is differentiable at $0$ and its Jacobian at $0$ is given by
  $$
  J[\vartheta^* \circ q^{[j]}_{0}](0) = J[\vartheta^*](q^{[j]}_{0}(0)) J[q^{[j]}_{0}](0) = J[\vartheta^*](q_0) (e_j - q_0).
  $$
\end{proof}

}

\change{
\begin{acknowledgements}
The authors would like to thank anonymous Referees for their constructive comments on an earlier version of this manuscript.
\end{acknowledgements}
}

\bibliographystyle{chicago}
\bibliography{biblio}

\end{document}

%% file: binom-chi-square.tex
\Sigma = \begin{bmatrix}
  0.010 & 0.049 & -0.059 \\ 
  0.049 & 0.246 & -0.295 \\ 
  -0.059 & -0.295 & 0.353 \\ 
   \end{bmatrix}
, \qquad
\Sigma_{n,N} = \begin{bmatrix}
  0.010 & 0.049 & -0.059 \\ 
  0.049 & 0.248 & -0.297 \\ 
  -0.059 & -0.297 & 0.356 \\ 
   \end{bmatrix}

%% file: moments-Hellinger.tex
\Sigma &= \begin{bmatrix}
  0.021 & -0.058 & 0.048 & -0.005 & -0.005 \\ 
  -0.058 & 0.168 & -0.154 & 0.037 & 0.007 \\ 
  0.048 & -0.154 & 0.176 & -0.080 & 0.011 \\ 
  -0.005 & 0.037 & -0.080 & 0.070 & -0.022 \\ 
  -0.005 & 0.007 & 0.011 & -0.022 & 0.009 \\ 
   \end{bmatrix}
, \\\Sigma_{n,N} &= \begin{bmatrix}
  0.021 & -0.059 & 0.048 & -0.006 & -0.005 \\ 
  -0.059 & 0.169 & -0.156 & 0.039 & 0.007 \\ 
  0.048 & -0.156 & 0.179 & -0.083 & 0.011 \\ 
  -0.006 & 0.039 & -0.083 & 0.072 & -0.022 \\ 
  -0.005 & 0.007 & 0.011 & -0.022 & 0.009 \\ 
   \end{bmatrix}

%% file: frechet-Sigma.tex
\Sigma = \begin{bmatrix}
  0.041 & -0.013 & -0.028 & -0.026 & 0.009 & 0.017 & -0.014 & 0.004 & 0.011 \\ 
  -0.013 & 0.059 & -0.047 & 0.009 & -0.037 & 0.028 & 0.003 & -0.023 & 0.019 \\ 
  -0.028 & -0.047 & 0.075 & 0.017 & 0.028 & -0.045 & 0.011 & 0.019 & -0.030 \\ 
  -0.026 & 0.009 & 0.017 & 0.034 & -0.011 & -0.023 & -0.007 & 0.002 & 0.006 \\ 
  0.009 & -0.037 & 0.028 & -0.011 & 0.039 & -0.028 & 0.002 & -0.002 & 0.001 \\ 
  0.017 & 0.028 & -0.045 & -0.023 & -0.028 & 0.051 & 0.005 & 0.001 & -0.006 \\ 
  -0.014 & 0.003 & 0.011 & -0.007 & 0.002 & 0.005 & 0.022 & -0.005 & -0.017 \\ 
  0.004 & -0.023 & 0.019 & 0.002 & -0.002 & 0.001 & -0.005 & 0.025 & -0.020 \\ 
  0.011 & 0.019 & -0.030 & 0.006 & 0.001 & -0.006 & -0.017 & -0.020 & 0.036 \\ 
   \end{bmatrix}

%% file: frechet-Sigma-n-N.tex
\Sigma_{n,N} = \begin{bmatrix}
  0.040 & -0.013 & -0.027 & -0.025 & 0.009 & 0.016 & -0.015 & 0.004 & 0.011 \\ 
  -0.013 & 0.059 & -0.046 & 0.010 & -0.037 & 0.027 & 0.003 & -0.022 & 0.019 \\ 
  -0.027 & -0.046 & 0.073 & 0.016 & 0.028 & -0.043 & 0.012 & 0.018 & -0.030 \\ 
  -0.025 & 0.010 & 0.016 & 0.033 & -0.011 & -0.022 & -0.007 & 0.001 & 0.006 \\ 
  0.009 & -0.037 & 0.028 & -0.011 & 0.039 & -0.028 & 0.002 & -0.002 & 0.000 \\ 
  0.016 & 0.027 & -0.043 & -0.022 & -0.028 & 0.050 & 0.006 & 0.001 & -0.006 \\ 
  -0.015 & 0.003 & 0.012 & -0.007 & 0.002 & 0.006 & 0.022 & -0.005 & -0.017 \\ 
  0.004 & -0.022 & 0.018 & 0.001 & -0.002 & 0.001 & -0.005 & 0.024 & -0.019 \\ 
  0.011 & 0.019 & -0.030 & 0.006 & 0.000 & -0.006 & -0.017 & -0.019 & 0.036 \\ 
   \end{bmatrix}